\documentclass[a4paper,11pt,reqno]{amsart}
\usepackage{amsfonts}
\usepackage{amsmath}
\usepackage{logicproof}
\usepackage{amssymb, amsthm}
\usepackage{tabularx}
\usepackage{caption}
\usepackage{booktabs}
\usepackage{multirow,dsfont}
\usepackage{array}
\usepackage{floatrow}
\usepackage{float}
\usepackage{floatpag}
\usepackage{centernot}
\usepackage[shortlabels]{enumitem}
\newcolumntype{L}[1]{>{\raggedright\let\newline\\\arraybackslash\hspace{0pt}}m{#1}}
\newcolumntype{C}[1]{>{\centering\let\newline\\\arraybackslash\hspace{0pt}}m{#1}}
\newcolumntype{R}[1]{>{\raggedleft\let\newline\\\arraybackslash\hspace{0pt}}m{#1}}
\usepackage[symbol]{footmisc}
\usepackage[labelfont=bf,format=plain,justification=raggedright,singlelinecheck=false]{caption}
\usepackage{mathrsfs}
\usepackage[]{mdframed}
\usepackage[colorlinks]{hyperref}
\usepackage{tcolorbox}
\usepackage{adjustbox}
\tcbuselibrary{theorems}
\newtcbtheorem[number within=section]{mytheo}{Note}
{colback=white!5,colframe=black!35!black,fonttitle=\bfseries}{th}

\numberwithin{equation}{section}
\usepackage{graphicx}
\usepackage{textgreek}

 {
      \theoremstyle{plain}
      
  }
\theoremstyle{plain}
\newtheorem{theorem}{Theorem}[section]
\newtheorem{corollary}[theorem]{Corollary}
\newtheorem{prop}[theorem]{Proposition}
\newtheorem{lemma}[theorem]{Lemma}
\theoremstyle{definition}
\newtheorem{definition}[theorem]{Definition}
\newtheorem{remark}[theorem]{Remark}

\def\CC{\mathbb C}

\def\ZZ{\mathbb Z}

\def\RR{\mathbb R}

\def\LLP{L^{\phi}}
\def\LPS{L^{\psi}}
\def\lphps{ L^{\phi,\psi}}
\begin{document}
\title[Relevant sampling of non-decaying signals.]{Relevant sampling of non-decaying signals in Orlicz and mixed-norm Orlicz spaces defined on a locally compact group}
\author[R. Sarthak]{Sarthak Raj}
\address[R. Sarthak]{Department of Mathematics, Indian Institute of Technology Delhi, India}
\email{sarthakraj.math@gmail.com}
\author{S. Sivananthan}
\address[S. Sivananthan]{Department of Mathematics, Indian Institute of Technology Delhi, India}
\email{siva@maths.iitd.ac.in}
\begin{abstract}
  We study the sampling problem for non-decaying signals defined on a locally compact group. The signals are modeled as elements of suitable subspaces of weighted Orlicz and weighted mixed-norm Orlicz spaces, thereby allowing controlled growth at infinity. Our principal focus is on image spaces of idempotent integral operators. Under suitable oscillation estimates on the associated integral kernels, we establish deterministic sampling theorems that guarantee stable reconstruction from pointwise samples. We further prove an average sampling result in the same setting. In addition, we prove random sampling theorems for functions whose norm is essentially concentrated on a compact subset of the group. In particular, we show that, with high probability, sampling is possible from
$\mathcal{O}\!\left(\mu(K)\log \mu(K)\right)$
randomly chosen sample points, where $K$ denotes the compact set on which the signal, belonging to Orlicz space, is essentially norm concentrated. In the same spirit, a random sampling result is proved for signals in mixed-norm Orlicz space. As applications of the abstract theory, we obtain sampling theorems for weighted shift-invariant spaces and Orlicz modulation spaces.
\end{abstract}
\maketitle
\section{Introduction}
Sampling theory is one of the cornerstones of signal processing. It provides a mathematical framework for recovering a signal from a discrete collection of its measurements. For a given function space $V$, the sampling problem consists of determining the existence of those discrete sets $X$ for which every function $f\in V$ is uniquely determined and can be stably reconstructed from its sample values
$\{f(x):x\in X\}.$ Thus, another fundamental goal of sampling theory is to characterize the sampling sets of $V$, namely those discrete sets that admit stable reconstruction. This problem has been extensively studied in harmonic analysis \cite{duffinSchaffer1952class,beurling1989contemporary,youngboook,olevskiiulnaoskii2016functionsdisconnectedspectrum}, approximation theory \cite{MhaskarFilbirMzManifold,GrochenigMZappquadrules,FilbirMZqSphere, Krieg_Ullrich_2026Samplingrecovery}, and learning theory \cite{CuckerSmale,SmaleZhouLearningSampling2004,SmaleZhouLearningSampling2005}. The first foundational result in the theory of modern sampling came from the seminal work of Whittaker, Kotelnikov, and Shannon \cite{Whittakersampling,Kotelnikovsampling,ShannonSampling}. They independently proved a sampling theorem for the space of band-limited functions, also known as the Paley-Wiener space. It is known as the \emph{Whittaker-Kotelnikov-Shannon Sampling Theorem}, and it states that any signal $f: \RR \to \CC$ band-limited to $[-\pi \omega, \pi \omega],$ for some $\omega >0,$ i.e., its Fourier transform $\hat{f}$ satisfies $\mathrm{supp}(\hat{f}) \subset [-\pi \omega, \pi \omega],$ can be recovered from its sample values $\{f(k/\omega): k \in \ZZ\}$ using the reconstruction formula
$$\displaystyle f(x) =  \sum_{k \in \ZZ} f(k/\omega)\frac{\sin \pi(\omega x - k)}{\pi(\omega x -k)} \quad \mbox{ for } x \in \RR.$$ From an engineering standpoint, the theorem states that a signal can be reconstructed without loss of information from its discrete samples, provided that the sampling frequency is more than twice the highest frequency component present in the signal. This critical threshold is known as the \emph{Nyquist rate}. Since then, much work has been done to generalize this result in multiple directions; see \cite{UnserSampling50}. On a different theme, a complete description of sampling sets for Paley--Wiener spaces in terms of \emph{Beurling density} \cite{beurling1989contemporary} was established by Landau in \cite{LandauDensity}. The remaining critical density case was later settled by Seip and Ortega-Cerdà \cite{SeipFourierFrame}.

However, the restriction of a signal being band-limited is too stringent for practical applications, as many naturally occurring signals are not band-limited \cite{Samplingnotbandlimited}. For this reason, sampling theory for a more general class of signals, namely shift-invariant spaces, was developed. From the WKS sampling theorem, one can see that the space of band-limited functions is also a shift-invariant space. A detailed account of sampling in shift-invariant spaces is given in \cite{AldroubiGrochenigSiam}. Other sampling problems studied in this setting include local reconstruction \cite{QSunLocalreconstruction}, convolution sampling \cite{AldroubiSUnConvolutionAverageSampling}, average sampling \cite{KangAverageSampling,GarciaAverageSamplinginHSOperators}, and sampling based on forward and backward differences \cite{GarciaDifferenceSampling}. A density theorem for sampling in shift-invariant spaces generated by totally positive functions of finite type was established in \cite{Grochenigtpfinite}. Subsequently, sampling in shift-invariant spaces generated by Mayer scaling function were proved in \cite{RadhaSelvanMeyerScaling}. A major breakthrough came in the work of Gr\"ochenig, Romero, and St\"ockler \cite{GrochenigRomeroStocklerSamplinginSIandGaborforTP}, where they prove density results for sampling in shift invariant spaces generated by functions of totally positive type. This setup was further generalized to the functions that belong to the image space of some idempotent integral operator, and sampling theorems were established \cite{NashedSunJFAIdempotent}. 

The above-mentioned works primarily focus on deterministic sampling sets. Another approach is to consider random sampling, where the sampling points are selected according to a prescribed probability distribution. This approach was introduced by Gr\"ochenig and Bass in \cite{GrochenigBassRandomMultiTrig} for subspaces of trigonometric polynomials. This received a considerable amount of attention due to its practical significance, and since then, the random sampling problem has been studied in multiple setups. The random sampling theorem for shift-invariant subspaces of Lebesgue spaces is explored in \cite{RandomSampSIYang,RelevantSamplingFUhr}. For the random sampling results on the image space of an idempotent integral operator on Lebesgue spaces, see \cite{QSunRandomACHA,DhirajSivaLp}. A generalization of the random sampling result to the image space of an idempotent integral operator on Orlicz spaces, which generalizes all the setups above, see \cite{PatelBajpeyiSivaRelevantOrlicz}. 

As it can be noted, membership in these spaces generally requires a certain degree of decay or localization of the functions. To accommodate signals that do not exhibit such decay, one is naturally led to consider weighted variants of these spaces. A foundational framework for the sampling of non-decaying signals was developed by Nguyen and Unser in \cite{UnserSamplingNOnDecaying}, where the authors studied sampling in weighted shift-invariant subspaces of Lebesgue spaces. This framework was subsequently extended to mixed-Lebesgue spaces in \cite{NondecayingMixedLebesgue}. Motivated by these works, we study the sampling problem for non-decaying signals in a more general setting. Specifically, we extend the framework to image spaces of idempotent integral operators acting on weighted Orlicz spaces and mixed-norm Orlicz spaces.
The sampling framework developed in \cite{UnserSamplingNOnDecaying,NondecayingMixedLebesgue} relies on a particular sampling procedure in which the signal is first pre-filtered by convolution with a suitable kernel, and the resulting filtered signal is then sampled. Equivalently, the measurements obtained in this manner coincide with the representation coefficients of the function in the underlying shift-invariant space. In contrast, our approach is more direct and relies on pointwise measurements of the signal over a discrete sampling set, without requiring any pre-filtering. Furthermore, we formulate our results for signals defined on a general locally compact group, in the spirit of the framework considered in \cite{FuhrGrochenigOscillationSampling}. 

In this paper, we study the sampling problem for non-decaying signals defined on a general locally compact group and belonging to the Orlicz space $L^{\phi},$ and its mixed-norm variant $L^{\phi,\psi},$ see Section \ref{prelim} for the relevant definitions. The choice of Orlicz spaces is justified by their flexibility to represent several function spaces, including the classical Lebesgue spaces. These spaces can accommodate signals that do not possess the standard $L^p$-integrability. By allowing arbitrary Orlicz functions in place of power functions, Orlicz spaces can model a wide range of growth and integrability behaviors, thereby providing a considerably richer framework than the classical $L^p$-case. A recent work in the setting of Orlicz spaces is the study of the sampling discretization problem for finite-dimensional subspaces, also known as the Marcinkiewicz--Zygmund inequalities, which is established in \cite{SergeyKOsovSamplingDiscretizationOrlicz}; see also \cite{TikhinivKolomoitsev2026marcinkiewicz}. This naturally motivates the study of sampling in the Orlicz space setting for more general infinite-dimensional subspaces. Mixed-norm Orlicz spaces, on the other hand, are considered to treat signals that depend on multiple variables, where each variable represents a unique characteristic of the signal. Thus, allowing separate Orlicz functions to govern these characteristics yields far-reaching generality. 

Our main motivation is to consider the sampling for non-decaying signals. The non-decaying nature of the signals is incorporated through the use of weight functions. More precisely, if $w$ is a growing weight, for instance a polynomial weight, then the signals are assumed to belong to the weighted spaces $L^{\phi}_{1/w}$ and $L^{\phi,\psi}_{1/w}$. The growth of $w$ allows functions in these spaces to exhibit controlled growth at infinity, and therefore provide a natural framework for modeling non-decaying signals. Our main results establish sampling theorems for image spaces of idempotent integral operators acting on these weighted spaces. Specifically, Theorems~\ref{samplingorlicz} and~\ref{samplingmixedorlicz} provide stable sampling theorems for the image spaces of idempotent integral operators on $L^{\phi}_{1/w}$ and $L^{\phi,\psi}_{1/w}$, respectively. We also establish analogous average sampling theorems in both settings. Finally, for a class of functions whose norm is essentially concentrated on a compact set, we prove a random sampling theorem. To demonstrate our theory, we consider shift-invariant spaces in both Orlicz and mixed-norm Orlicz spaces, as well as Orlicz-modulation and mixed-norm Orlicz-modulation spaces.  

\subsection{Main Contributions}

The principal contributions of this paper are summarized below.

\begin{enumerate}
    \item We establish sampling theorems for non-decaying signals defined on locally compact groups and belonging to the image spaces of idempotent integral operators acting on weighted Orlicz and weighted mixed-norm Orlicz spaces- Theorem \ref{samplingorlicz} and \ref{samplingmixedorlicz}. The results are obtained under suitable oscillation estimates on the kernels of the integral operators. As an essential part of the proofs, we establish norm equivalence estimates between the underlying function spaces and the associated sequence spaces- Lemma \ref{seqequiv} and \ref{mixedseqequiv}.

    \item We extend the above framework to the setting of average sampling and establish corresponding average sampling theorems for both weighted Orlicz and weighted mixed-norm Orlicz spaces- Theorem \ref{avgsamplingorlicz} and \ref{avgsamplingmixedorlicz}.

    \item For functions in the Orlicz space, whose norm is essentially concentrated on the compact ball $G_N$, we prove a random sampling theorem. In particular, we show that, with high probability, stable recovery is possible from  $\mathcal{O}\left(\mu\left(G_N\right)\log\mu\left(G_N\right) \right)$ randomly chosen sample points. We also establish the corresponding random sampling theorem in the weighted mixed-norm Orlicz setting.
    \item In \cite{UnserSamplingNOnDecaying,NondecayingMixedLebesgue}, Riesz-type bounds were established to guarantee the well-definedness of certain shift-invariant subspaces of Lebesgue and mixed-Lebesgue spaces. We extend these results to the setting of weighted Orlicz and weighted mixed-norm Orlicz spaces by proving analogous Riesz-type bounds. These spaces turn out to be the image space of idempotent integral operators and are therefore used to demonstrate the sampling results. Other objects for demonstration of sampling results are the Orlicz and mixed-norm Orlicz modulation spaces.
    \item As auxiliary results of independent interest, we establish boundedness properties of idempotent integral operators on weighted Orlicz and weighted mixed-norm Orlicz spaces.
\end{enumerate} 
\subsection{Organization} This paper is organized as follows. In Section~\ref{prelim}, we introduce the necessary background and notation. Section~\ref{estimates} is devoted to establishing the fundamental estimates that are used repeatedly throughout the remainder of the paper. In Sections~\ref{samplingorliczsection} and~\ref{samplingmixedorliczsection}, we prove sampling theorems for the image spaces of idempotent integral operators on weighted Orlicz and weighted mixed-norm Orlicz spaces, respectively. Finally, in Section~\ref{examples}, we establish the well-definedness of certain shift-invariant subspaces of weighted Orlicz and weighted mixed-norm Orlicz spaces, and apply the developed theory to obtain sampling theorems for these spaces as well as for Orlicz and mixed-norm Orlicz modulation spaces. 

\section{Preliminaries}\label{prelim}
This section is devoted to introducing the notation and mathematical background that will be used throughout the paper.

 Let $G$ be a locally compact, Hausdorff, second countable group. It is well known that $G$ admits a unique (up to scalar multiples) left-invariant Radon measure, called the \emph{left Haar measure}, which will be denoted by $\mu$. The modular function on $G$ will be denoted by $\Delta.$ Further, $G$ is metrizable. Moreover, the metric $d$ can be chosen to be left-invariant and every closed ball is compact \cite{haagerup2006proper}.
\begin{definition}
    A countable subset $\{x_i : i \in I\} \subset G$ is called relatively separated if there is a $r>0$ such that 
    \begin{equation*}
        \sup_{x \in G} \sum_{i \in I} \chi_{B(x_i,r)}(x) < \infty.
    \end{equation*}
    And it is called $q-$dense if 
    \begin{equation*}
         \inf_{x \in G} \sum_{i \in I} \chi_{B(x_i,q)}(x) \geq 1.
    \end{equation*}
\end{definition}
For a relatively separated set $X=\{x_i:i\in I\}$, we define
$$N_X(x):=\sum_{i\in I}\chi_{B(x_i,r)}(x)$$
and
$$L_X:=\sup_{x\in G}N_X(x).$$

The class of non-decaying signals studied in this article is taken to consist of functions belonging to Orlicz spaces and mixed-norm Orlicz spaces. The decay of these functions are controlled by the weight on these spaces. We now define these spaces and fix the notations. Our presentation follows the definitions, notation, and terminology of \cite{RaoRenOrliczbook}.
\begin{definition}[Orlicz function]
    A continuous convex non-decreasing function $\phi: \RR \to [0, \infty]$ is said to be an Orlicz function (or Young's function) if
    \begin{enumerate}[a)]
    \item $\phi(x)=\phi(-x),$
        \item $\phi(0)=0$ and $\lim_{x \to \infty} \phi(x)= \infty.$
    \end{enumerate}
\end{definition}
In the classical definition, an Orlicz function is not required to be continuous. Indeed, it may attain the value infinity, and consequently may fail to be continuous. Throughout this paper, however, we restrict our attention to continuous Orlicz functions. \\
The following are some standard examples of Orlicz functions:
\begin{enumerate}
    \item $\phi(x)=|x|^p /p $ for $p \geq 1$. 
    \item $\phi(x)=e^{|x|}-|x|-1$.
    \item $\phi(x)=|x|\log(|x|)$.
    \item $\phi_{\alpha}(x)=|x|^{\alpha} (1+\log(|x|))$ for $\alpha >3$.
\end{enumerate}
For an Orlicz function $\phi,$ a related convex function $\phi^*,$ which is referred to as \emph{complementary function} to $\phi$, can be defined as 
\begin{equation}
\label{complementaryorlicz}
    \phi^{*}(y)=\sup\{x|y|- \phi(x): x \geq 0\}.
\end{equation} \par
It is easily seen that this function is again an Orlicz function. The pair $(\phi,\phi^*)$ of an Orlicz function and its complementary function is called a \emph{complementary Young's pair}, and it satisfies the Young's inequality 
\begin{equation}
\label{orliczholder}
    xy \leq \phi(x) + \phi^*(y) \mbox{ for } x,y \in \RR
\end{equation}
which is evident from (\ref{complementaryorlicz}).
\begin{definition}[$\Delta_2-$condition]
    An Orlicz function $\phi$ is said to satisfy the $\Delta_2-$condition if there is a constant $C>0$ such that 
    $$\phi(2x) \leq C \phi(x) \mbox{ for all } x>0. $$
\end{definition}

\begin{definition}[Weight function]
    A continuous function
    $$w: G \to (0, \infty)$$ 
    satisfying $w(x)=w(x^{-1})$ for all $x \in G$ is called a weight function.
\end{definition}
A weight function is said to be sub-multiplicative if it satisfies 
$$w(xy) \leq w(x) w(y)$$ for all $x$ and $y$ in $G.$

\begin{definition}[Weighted Orlicz space]
     Let $\phi$ be an Orlicz function, then the associated Orlicz space is denoted by $L^{\phi}(G),$ and is given as
    \begin{equation*}
        L^{\phi}(G):=\left\{f: G \to \CC \mbox{ measurable }: \int_{G} \phi (\lambda |f(x)|)~ d\mu(x) < \infty \mbox{ for some } \lambda > 0 \right\}.
    \end{equation*}
   The space $L^{\phi}(G)$ can be endowed with a norm, called the \emph{Luxemburg} norm (or \emph{Gauge} norm), given by
   \begin{equation*}
    \|f\|_{L^{\phi}(G)}:= \inf \left\{ \lambda > 0: \int_{G} \phi\left(  \frac{|f(x)|}{\lambda}\right) ~ d\mu(x) \leq 1 \right\}
  \end{equation*}
  which makes the space $(L^{\phi}(G),\|\cdot\|_{L^{\phi}(G)})$ a Banach space.\\ 
     Let  $w$ be a weight function, then the $w-$weighted Orlicz space $L^{\phi}_{w}(G)$ is defined as
    \begin{equation*}
        L^{\phi}_{w}(G)= \left \{ f: G \to \CC \mbox{ measurable }: fw \in L^{\phi}(G) \right\}
    \end{equation*}
    and the norm is given by $\|f\|_{L^{\phi}_{w}(G)}=\|fw\|_{L^{\phi}(G)}.$
\end{definition}

    If an Orlicz function $\phi$ satisfies the $\Delta_2-$condition, then the set of all compactly supported continuous functions $C_{c}(G)$ is dense in $L^{\phi}(G).$ Further, if $(\phi,\phi^*)$ is a complementary Young's pair and $\phi$ satisfies the $\Delta_2-$condition, then $(L^{\phi}(G), \|\cdot\|_{L^{\phi}(G)})^* = (L^{\phi^*}(G), \|\cdot\|_{L^{\phi^*}(G)}).$ If $\phi^*$ also satisfies the $\Delta_2-$condition, then the Banach space $(L^{\phi}(G), \|\cdot\|_{L^{\phi}(G)})$ is reflexive.

   Throughout the remainder of this paper, we assume that every Orlicz function satisfies the $\Delta_2-$condition.

    Let $ (\phi,\phi^*) $ be a complementary Young's pair. Then \emph{H\"older's inequality} for the Orlicz space is given as
 \begin{equation}
     \|fg\|_{L^1(G)} \leq 2 \|f\|_{L^{\phi}(G)} \|g\|_{L^{\phi^*}(G)}
 \end{equation}
 for $f \in L^{\phi}(G)$ and $g \in L^{\phi^*}(G).$

We now introduce the sequence spaces used to model sample values of the signals. These spaces are discrete analogs of the corresponding function spaces introduced above.
\begin{definition}[Orlicz sequence space]
    Let $\Lambda$ be a discrete subset of $G.$ Then, the Orlicz sequence space $\ell^{\phi}(\Lambda)$, is defined as 
    \begin{equation*}
        \ell^{\phi}(\Lambda):=\left\{ (c_{\gamma})_{\gamma \in \Lambda} \in \CC^{\Lambda}: \sum_{\gamma \in \Lambda} \phi\left( {\lambda|c_{\gamma}|} \right)< \infty \mbox{ for some } \lambda > 0 \right\}.
    \end{equation*}
    The corresponding sequence norm is defined as 
\begin{equation*}
    \|(c_\gamma)_{\gamma \in \Lambda}\|_{\ell^{\phi}(\Lambda)}:=\inf \left\{ \lambda> 0: \sum_{\gamma \in \Lambda} \phi \left( \frac{|c_{\gamma}|}{\lambda}\right) \leq 1 \right\}.
\end{equation*} \\
Let $w$ be a weight function, then the $w-$weighted Orlicz sequence space $\ell^{\phi}_{w}(\Lambda)$ is defined as
    \begin{equation*}
        \ell^{\phi}_{w}(\Lambda):=\left\{ (c_{\gamma})_{\gamma \in \Lambda}: (c_{\gamma} w(\gamma))_{\gamma \in \Lambda} \in \ell^{\phi}(\Lambda) \right\}
    \end{equation*}
    with the the norm given by $\|(c_{\gamma})_{\gamma\in \Lambda}\|_{\ell^{\phi}_{w}(\Lambda)}=\|(c_{\gamma} w(\gamma))_{\gamma \in \Lambda}\|_{\ell^{\phi}_{}(\Lambda)}.$
\end{definition}

Next, we are going to introduce the weighted mixed-norm Orlicz spaces. Here the weight will be a continuous function on $G \times G,$ with sub-multiplicativity defined as usual.
\begin{definition}[Mixed norm Orlicz spaces]
    Let $\phi, \psi$ be two Orlicz functions. Then the mixed norm Orlicz space $L^{\phi,\psi}(G \times G)$ is defined as
    \begin{equation*} 
        L^{\phi,\psi}(G \times G) := \left\{ f \in \mathcal{F}(G,L^\psi(G)):  \int_{G} \psi\left(\lambda \|f(\cdot, y)\|_{L^{\phi}(G)}\right)~ d\mu(y) < \infty \mbox{ for some } \lambda >0 \right\},
    \end{equation*}
    where $\mathcal{F}(G,L^\psi(G))$ denotes the collection of measurable functions $f: G \to L^{\psi}(G)$. \\ \\
    \textbf{A word of caution:} It is customary to define mixed-norm Orlicz spaces, and more generally Banach-valued function spaces, using strongly measurable functions. In our setting, however, it is sufficient to assume measurability. Indeed, the range space $L^{\psi}(G)$ is separable, since the Orlicz function $\psi$ satisfies the $\Delta_2$-condition. Consequently, by the Pettis Measurability Theorem, measurability and strong measurability are equivalent, and therefore there is no need to distinguish between the two notions.
    
    The corresponding norm for $f \in L^{\phi,\psi}(G \times G)$ is given by 
\begin{equation*}
    \|f\|_{\lphps(G \times G)}:=\| y \mapsto \|f(\cdot,y)\|_{\LLP(G)} \|_{\LPS(G)}= \inf \left\{ \lambda > 0: \int_{G} \psi \left( \frac{\|f(\cdot,y)\|_{\LLP(G)}}{\lambda} \right)d\mu(y) \leq 1  \right\}
\end{equation*}
which makes it a Banach space.\\
Let $w$ be a weight function on $G \times G$, then the $w-$weighted mixed norm Orlicz space $L^{\phi,\psi}_{w}(G\times G)$ is defined as
    \begin{equation*}
        L^{\phi,\psi}_{w}(G\times G)= \left \{ f: G \times G \to \CC \mbox{ measurable }: fw \in L^{\phi,\psi}(G \times G) \right\}
    \end{equation*}
    and the norm is given by $\|f\|_{L^{\phi,\psi}_{w}(G \times G)}=\|fw\|_{L^{\phi,\psi}(G\times G)}.$
\end{definition}
 
  Let $\phi, \psi$ be Orlicz functions and $\phi^*,\psi^*$ be their complementary Orlicz functions, respectively. Then, the H\"older's inequality in this case is given by 
\begin{equation}
    \int_{G} \int_{G} |fg| \leq 4 \|f\|_{L^{\phi, \psi}(G \times G)} \|g\|_{L^{\phi^*, \psi^*}(G \times G)}
\end{equation}
for $f \in L^{\phi, \psi}(G \times G),$ and $g \in L^{\phi^*, \psi^*}(G \times G).$ If $\phi,\psi$ are strictly convex, then the duality relation $L^{\phi, \psi}(G \times G)^* = L^{\phi^*, \psi^*}(G \times G)$ also holds.   

\begin{definition}[Mixed norm Orlicz sequence spaces]
     Let $\phi, \psi$ be two Orlicz functions, and $\Lambda$ be a discrete subset of $G$. Then the mixed norm Orlicz sequence space $\ell^{\phi,\psi}(\Lambda \times \Lambda)$ is defined as
    \begin{equation*}
        \ell^{\phi,\psi}(\Lambda \times \Lambda) := \left\{ (c_{\gamma, \gamma'})_{(\gamma,\gamma') \in \Lambda \times \Lambda} \in \CC^{\Lambda \times \Lambda}:  \sum_{\gamma' \in \Lambda} \psi\left(\lambda \|(c_{\gamma, \gamma'})_{\gamma \in \Lambda}\|_{\ell^{\phi}(\Lambda)}\right) < \infty \mbox{ for some } \lambda >0 \right\}.
    \end{equation*}
    The corresponding sequence norm for $(c_{\gamma, \gamma'})_{(\gamma,\gamma') \in \Lambda \times \Lambda} \in \ell^{\phi,\psi}(\Lambda \times \Lambda)$ is given by 
\begin{equation*}
    \|(c_{\gamma, \gamma'})_{(\gamma,\gamma') \in \Lambda \times \Lambda}\|_{\ell^{\phi,\psi}(\Lambda \times \Lambda)}:= \left\|\left( \| (c_{\gamma, \gamma'})_{\gamma \in \Lambda}\|_{\ell^{\phi}(\Lambda)} \right)_{\gamma' \in\Lambda}  \right\|_{\ell^{\psi}(\Lambda)}.
\end{equation*}
The $w-$weighted mixed norm Orlicz sequence space $\ell^{\phi,\psi}_{w}(\Lambda \times \Lambda)$ is defined as
    \begin{equation*}
        \ell^{\phi,\psi}_{w}(\Lambda \times \Lambda):=\left\{ (c_{\gamma,\gamma'})_{(\gamma,\gamma') \in \Lambda \times \Lambda}: (c_{\gamma,\gamma'} w(\gamma,\gamma'))_{(\gamma,\gamma') \in \Lambda \times \Lambda} \in \ell^{\phi,\psi}(\Lambda\times \Lambda) \right\}
    \end{equation*}
    with the norm given by $\|(c_{\gamma,\gamma'})_{(\gamma,\gamma')\in \Lambda \times \Lambda}\|_{\ell^{\phi,\psi}_{w}(\Lambda \times \Lambda)}=\|(c_{\gamma,\gamma'} w(\gamma,\gamma'))_{(\gamma,\gamma') \in \Lambda \times \Lambda}\|_{\ell^{\phi,\psi}_{}(\Lambda \times \Lambda)}.$
\end{definition}

\begin{remark}
It is worth emphasizing that the Orlicz space $L^\phi(G\times G)$ is, in general, different from the mixed-norm Orlicz space $L^{\phi,\phi}(G\times G)$, even when the same Orlicz function $\phi$ is used in both variables. The former is defined with respect to the product measure on $G\times G$, while the latter is obtained by iteratively applying the Luxemburg norm in each variable. Consequently, the two spaces are generally distinct and their norms are not, in general, equivalent.
\end{remark}
The following notion of oscillation will play an important role throughout the paper. For a function $f$ on $G$, the $r$-oscillation of $f$ at a point $x\in G$, denoted by $\operatorname{osc}_r(f)(x)$, is defined by
$$\operatorname{osc}_r(f)(x)=\sup_{d(x,y)<r}|f(y)-f(x)|.$$
For functions of several variables, the notion of oscillation is defined analogously by taking the supremum over each individual coordinates lying within distance $r$ of the corresponding coordinates of the given point.

In the coming sections, $a \lesssim b $ means that there is a constant c, independent of relevant parameters, such that $a \leq cb.$ When $a\lesssim b$ and $b \lesssim a,$ we write $a \simeq b$.
 \section{Estimates} \label{estimates}
  This section is devoted to: (i) proving certain equivalences between sequence norms and the corresponding function norms, (ii) proving boundedness of the integral operators. These estimates play a central role in the proofs of the sampling theorems and will be used repeatedly in the sections that follow.

 We begin by establishing norm equivalences that allow sequence norms to be expressed in terms of the corresponding function norms. The first such result is given below.
\begin{lemma}
\label{seqequiv}
    Let $X=\{x_i: i \in I\} \subset G$ be a relatively separated and $r-$dense subset.
    Then for a sequence $(c_i)_{i \in I}$ we have 
    \begin{equation}\label{seqequiveqn}
    \|(c_i)\|_{l^{\phi}_{1/w}(X)} \simeq  \|\sum_{i \in I} |c_i| \chi_{B(x_i,r)}\|_{\LLP_{1/w}(G)},    
    \end{equation}
    where $w$ is a sub-multiplicative weight.
\end{lemma}
\begin{proof}
Let $X=\{x_i: i \in I\}$ be a relatively separated and $r-$dense set. We begin the proof by first considering the following 
\begin{align*}
    \sum_{i \in I} \phi \left( \frac{|c_i|}{w(x_i) \lambda }\right) \mu(B(x_i,r)) &= \sum_{i \in I} \phi \left( \frac{|c_i|}{w(x_i) \lambda }\right) \int_{G} \chi_{B(x_i,r)}(x) ~ d\mu(x) \\
    (*) &= \sum_{i \in I} \int_{G}\phi \left( \frac{|c_i|}{w(x_i) \lambda }\right) \chi_{B(x_i,r)}(x) ~ d\mu(x).
\end{align*}
Since the set  $X$ is relatively separated, therefore each point $x \in G$ can be atmost in $L_X$ balls from $\{B(x_i,r): i \in I\},$ and so
\begin{align*}
     \hspace{4cm} (*) &\leq L_X \int_{G} \phi \left( \sum_{i \in I}\frac{|c_i|}{w(x_i) \lambda }\chi_{B(x_i,r)}(x)\right) ~ d\mu(x).
\end{align*}
For $x \in B(x_i,r)=x_i B(e,r),$ $x_i ^{-1} x \in B(e,r)$ and by the sub-multiplicativity of $w$ 
$$w(x) = w(x_i x_i^{-1} x)\leq w(x_i) w(x_{i} ^{-1} x).$$
Put $M_{\infty}=\sup_{t \in B(e,r)} w(t),$ then 
$$w(x) \leq M_{\infty} w(x_i).$$
Substituting this above, we get 
\begin{align*}
    (*) &\leq L_X \int_{G} \phi \left(  \frac{\sum_{i \in I}M_{\infty}|c_i| \chi_{B(x_i,r)}(x)}{ \lambda w(x)} \right) d \mu(x).
\end{align*}
Therefore 
$$\|\left( c_{i} \right)\|_{\ell^{\phi}_{1/w}(X)} \lesssim L_X  M_{\infty} \left\|\sum_{i \in I} |c_i| \chi_{B(x_i,r)} \right\|_{\LLP_{1/w}(G)}.$$
Multiplying the sum by a constant, as done above, changes the sequence norm up to an absolute constant, which depends only on $r$. To avoid cumbersome notation, we don't write it explicitly. \par
For the reverse estimate,
\begin{align*}
     \int_{G}\phi \left( \frac{\sum_{i \in I} |c_i| \chi_{B(x_i,r)}(x)}{\lambda w(x) }\right) ~ d\mu(x) &= \int_{G}\phi \left( \sum_{i \in I} \frac{ |c_i| }{\lambda  } \frac{N_{X}(x)}{w(x)} \frac{\chi_{B(x_i,r)}(x)}{N_{X}(x)}\right) ~ d\mu(x) \\
     \mbox{(convexity and monotonicity)}  &\leq \int_{G} \sum_{i \in I} \phi\left( L_{X} \frac{|c_i|}{\lambda w(x)} \right) \chi_{B(x_i,r)}(x) ~ d\mu(x)
\end{align*}
again 
$$w(x_i)=w(x x^{-1} x_i) \leq w(x) w(x^{-1} x_i) \leq M_{\infty} w(x)$$
and so 
\begin{align*}
     \int_{G}\phi \left( \frac{\sum_{i \in I} |c_i| \chi_{B(x_i,r)}(x)}{\lambda w(x) }\right) ~ d\mu(x) &\leq \sum_{i \in I} \phi\left( L_{X}M_{\infty} \frac{ |c_i| }{w(x_i) \lambda} \right) \mu(B(x_i,r)). 
\end{align*}
Thus 
\begin{align*}
    \left\|\sum_{i \in I} |c_{i}| \chi_{B(x_i,r)}\right\|_{\LLP_{1/w}(G)} \lesssim L_{X}M_{\infty} \|(c_{i})\|_{\ell^{\phi}_{1/w}(X)}.
\end{align*}
\end{proof}
The next lemma proves a similar result for the weighted mixed-norm Orlicz spaces.
\begin{lemma}
\label{mixedseqequiv}
      Let $X=\{x_i: i \in I\} \subset G$ be a countable set which is relatively separated and $r-$dense.
    Then for a sequence $(c_{i,j})_{i,j \in I}$ we have 
    \begin{equation} \label{mixedseqequiveqn}
    \|(c_{i,j})\|_{l^{\phi,\psi}_{1/w} (X \times X)} \simeq \left\|\sum_{i,j \in I} |c_{i,j}| \left(\chi_{B(x_i,r)}\otimes \chi_{B(x_j,r)}\right) \right\|_{\lphps_{1/w}(G \times G)},    
    \end{equation}
    where $w$ is a sub-multiplicative weight and $g \otimes h (x,y)=g(x)h(y)$.
\end{lemma}
\begin{proof}
    We prove this for the non-weighted case, and then we will adapt the result to the weighted case. For the proof, we use the equivalence described above for the Orlicz space. Since the mixed norm is defined as the consecutive application of the Orlicz norm, we use the equivalence in two steps, one for each application of the Orlicz norm.\par 
Let $X=\{x_i : i \in I\}$ be a relatively separated and $r-$dense subset of $G.$ Note that for each fixed $j,$ we have 
$$\|(c_{i,j})_{i \in I}\|_{\ell^{\phi}(X)} \lesssim L_{X} \left \| \sum_{i \in I} |c_{i,j}| \chi_{B(x_i,r)} \right \|_{L^{\phi}(G)}.$$
Now using this we get 
\begin{align*}
    \sum_{j \in I} \|(c_{i,j})_{i \in I}\|_{\ell^{\phi}(X)} \chi_{B(x_j,r)}(y) &\lesssim L_{X} \sum_{j \in I} \left \| \sum_{i \in I} |c_{i,j}| \chi_{B(x_i,r)} \right \|_{L^{\phi}(G)} \chi_{B(x_j,r)}(y) \\
    &\lesssim L_{X}^2 \left \|\sum_{j \in I} \sum_{i \in I} |c_{i,j}| \chi_{B(x_i,r)}(\cdot) \chi_{B(x_j,r)}(y) \right \|_{L^{\phi}(G)}.
\end{align*}
This implies 
\begin{align*}
    \left \|\sum_{j \in I} \|(c_{i,j})_{i \in I}\|_{\ell^{\phi}(X)} \chi_{B(x_j,r)} \right \|_{L^{\psi}(G)} \lesssim L_{X}^2 \left \| \left\| \sum_{i,j \in I} |c_{i,j}| \chi_{B(x_i,r)}(\cdot)\chi_{B(x_j,r)}(y)\right\|_{L^{\phi}(G)} \right\|_{L^{\psi}(G)}.
\end{align*}
But also one has that 
\begin{align*}
    \left \|\sum_{j \in I} \|(c_{i,j})_{i \in I}\|_{\ell^{\phi}(X)} \chi_{B(x_j,r)} \right \|_{L^{\psi}(G)} \gtrsim \frac{1}{L_X} \left\| \left(\|(c_{i,j})_{i \in I}\|_{\ell^{\phi}(X)} \right)_{j \in I} \right\|_{\ell^{\psi}(X)} =\frac{1}{L_X} \|(c_{i,j})_{i,j \in I}\|_{\ell^{\phi,\psi}(X \times X)},
\end{align*}
so
\begin{align*}
    \|(c_{i,j})_{i,j \in I}\|_{\ell^{\phi,\psi} (X \times X)} \lesssim L_{X}^3 \left\| \sum_{i,j \in I} |c_{i,j}| \left(\chi_{B(x_i,r)}\otimes \chi_{B(x_j,r)}\right)\right\|_{L^{\phi,\psi}(G \times G)}.
\end{align*}
The following trail of inequalities proves the other direction of the equivalence:
\begin{align*}
    \left\| \sum_{i,j \in I} |c_{i,j}| \left(\chi_{B(x_i,r)}\otimes \chi_{B(x_j,r)}\right)\right\|_{L^{\phi,\psi}(G \times G)} &= \left \| \left\| \sum_{i,j \in I} |c_{i,j}|  \chi_{B(x_i,r)}(\cdot) \chi_{B(x_j,r)}(y) \right\|_{L^{\phi}(G)} \right\|_{L^{\psi}(G)} \\
 (\mbox{Trinagle inequality})   & \lesssim \left \| \sum_{j \in I}\left\| \sum_{i \in I} |c_{i,j}| \chi_{B(x_i,r)} \right\|_{L^{\phi}(G)} \chi_{B(x_j,r)}\right\|_{L^{\psi}(G)} \\
    & \lesssim L_{X} \left \| \sum_{j \in I} \|(c_{i,j})_{i \in I}\|_{\ell^{\phi}(X)} \chi_{B(x_j,r)} \right\|_{L^{\psi}(G)} \\
    &\lesssim L_{X}^2 \left\|(c_{i,j})_{i,j \in I} \right\|_{\ell^{\phi,\psi}(X \times X)}.
\end{align*}
Now, to get the weighted estimates, observe that 

\begin{multline*}\frac{1}{L_{X}^2} \left\| \sum_{i,j \in I} \frac{|c_{i,j}|}{w(x_i,x_j)} \left(\chi_{B(x_i,r)}\otimes \chi_{B(x_j,r)}\right)\right\|_{L^{\phi,\psi}(G \times G)} \lesssim \left\|\left(\frac{c_{i,j}}{w(x_i,x_j)}\right)_{i ,j \in I}\right\|_{\ell^{\phi,\psi}(X \times X)}\\ \lesssim L_{X}^3 \left\| \sum_{i,j \in I} \frac{|c_{i,j}|}{w(x_i,x_j)} \left(\chi_{B(x_i,r)}\otimes \chi_{B(x_j,r)}\right)\right\|_{L^{\phi,\psi}(G \times G)}.\end{multline*}
Putting
$$\widetilde{M}_{\infty}= \sup\{ w(x,y): {d(x,e)<r, d(y,e) <r} \}< \infty,$$
and noting that for $x \in B(x_i,r)$ and $y \in B(x_j,r),$ we have 
$$w(x_i,x_j) \leq \widetilde{M}_{\infty} w(x,y) \mbox{ and } w(x,y) \leq \widetilde{M}_{\infty} w(x_i,x_j)$$
by the sub-multiplicativity of the weight function. Using these relations in the above equivalence gives 
\begin{multline*}\frac{1}{\widetilde{M}_{\infty}L_{X}^2} \left\| \sum_{i,j \in I} |c_{i,j}| \left(\chi_{B(x_i,r)}\otimes \chi_{B(x_j,r)}\right)\right\|_{L^{\phi,\psi}_{1/w}(G \times G)} \lesssim \left\|\left(c_{i,j}\right)_{i ,j \in I}\right\|_{\ell^{\phi,\psi}_{1/w}(X \times X)}\\ \lesssim \widetilde{M}_{\infty}L_{X}^3 \left\| \sum_{i,j \in I} |c_{i,j}|\left(\chi_{B(x_i,r)}\otimes \chi_{B(x_j,r)}\right)\right\|_{L^{\phi,\psi}_{1/w}(G \times G)}.\end{multline*}

\end{proof}
We now turn our attention to establishing the boundedness of the integral operators. These results rely on the assumptions imposed on the kernel of the integral operator. To this end, we first require an interpolation theorem for operators on Orlicz spaces, which we state below.
\begin{theorem}[Orlicz Interpolation Theorem, \cite{OrliczInterpolation}]\label{orliczinterpolation}
    Let $T: L^{1}(G) + L^{\infty}(G) \to L^{1}(G) + L^{\infty}(G) $ be an operator that maps $T$ maps $L^{1}(G)$ into $L^1 (G)$ and $L^\infty (G)$ into $L^\infty (G)$ and 
    \begin{align*}
        \|Tf-Tg\|_{L^1 (G)} &\leq M \|f-g\|_{L^1 (G)} \quad \forall f,g \in L^{1} (G)  \\
        \|Tf\|_{L^{\infty}(G)} &\leq M \|f\|_{L^{\infty} (G)} \quad \forall f \in L^{\infty} (G).
    \end{align*}
    Then $T$ maps $L^{1}(G) \cap L^{\phi} (G)$ into $L^{\phi} (G)$ and 
    \begin{equation*}
        \|Tf\|_{L^{\phi}(G)} \leq M \|f\|_{L^{\phi}(G)}, \quad \forall f \in L^{1} (G) \cap L^{\phi} (G).
    \end{equation*}
\end{theorem}
The following result is motivated by the corresponding result for Lebesgue spaces; see \cite{DahlkeShearletCoorbit}.
\begin{lemma}[Weighted Schur's Lemma]\label{weightedorliczschur}
Let $K$ be a Borel measurable function on $G \times G$, and $w$ be a weight function on $G$. Consider the following integral   
    \begin{equation*}
        Tf(x)=\int_{G} K(x,y) f(y) ~ d\mu(y)
    \end{equation*}
for $f \in \LLP_{1/w}(G) $. If the integral kernel $K$ satisfies 
\begin{equation}
    \|K\|_{S}= \operatorname{max}\left\{ \sup_{x \in G} \int_{G} |K(x,y)| \frac{w(y)}{w(x)} ~ d \mu(y), \sup_{y \in G} \int_{G} |K(x,y)| \frac{w(y)}{w(x)} ~ d\mu(x)\right\} < \infty,
\end{equation}
then, $Tf(x)$ is defined for $\mu-$a.e. $x$ and $T$ is a bounded operator on $\LLP_{1/w}(G).$ Moreover, $T$ satisfies the estimate
$$\|Tf\|_{\LLP_{1/w}(G)} \leq \|K\|_{S} \|f\|_{\LLP_{1/w}(G)}, \quad \forall f \in L^{\phi}_{1/w}(G).$$
\end{lemma}
\begin{proof}
    The result follows directly, for non-weighted case, from a single application of the Orlicz interpolation theorem, Theorem \ref{orliczinterpolation}, once we observe that
    \begin{equation*}
    \begin{split}
        \|Tf  \|_{L^{1}(G)} &\leq \|K\|_{S} \|f \|_{L^1 (G)} \mbox{ for } f \in L^{1}(G), \mbox{ and } \\
        \|Tf \|_{L^{\infty}(G)} &\leq \|K\|_{S} \|f  \|_{L^{\infty} (G)} \mbox{ for } f \in L^{\infty}(G).
    \end{split}
    \end{equation*}
    These bounds follow immediately from the definition of the operator. Now for the non-weighted case, observe that
    \begin{equation*}
        \frac{Tf(x)}{w(x)} = \int_{G} K(x,y) \frac{w(y)}{w(x)} \frac{f(y)}{w(y)} d\mu(y).
    \end{equation*}
    And therefore, 
    \begin{equation*}
        \|Tf\|_{\LLP_{1/w}(G)} \leq \|K\|_{S} \|f\|_{\LLP_{1/w}(G)}
    \end{equation*}
    follows by applying the non-weighted case to the kernel $\displaystyle\widetilde{K}(x,y)=K(x,y) \frac{w(y)}{w(x)}.$
\end{proof}
The weighted Schur's lemma yields, as a corollary, a generalized Young's inequality for convolution. 
\begin{corollary}[Generelazied Young's Inequality for Orlicz Spaces]\label{youngorlicz}
Let $w$ be a sub-multiplicative weight, and $\phi$ be an Orlicz function satisfying the $\Delta_{2}-$condition. For $f \in L^{\phi}_{1/w}(G)$ and $h \in L^{1}_{w}(G) \cap L^{1}_{w\Delta^{-1}}(G),$ we have the following 
$$ \|f * h\|_{L^{\phi}_{1/w}(G)} \leq \|f\|_{L^{\phi}_{1/w}(G)} \max\{\|h\|_{L^{1}_{w}(G)}, \|h\|_{L^{1}_{w\Delta^{-1}}(G)}\}.$$
\end{corollary}
\begin{proof}
   Applying the weighted Schur's lemma to the kernel $K(x,y)=h(y^{-1}x),$ we get 
   \begin{equation*}
       \begin{split}
            \int_{G} |h(y^{-1}x)| \frac{w(y)}{w(x)}d\mu(y) &\leq \int_{G} |h(y^{-1}x)| w(y^{-1}x) d\mu(y) \\
            & = \int_{G} |h(u)| w(u) \Delta^{-1}(u) du =\|h\|_{L^{1}_{w \Delta^{-1}}(G)}.
       \end{split}
   \end{equation*}
   Also, 
   \begin{equation*}
       \begin{split}
            \int_{G} |h(y^{-1}x)| \frac{w(y)}{w(x)}d\mu(x) &\leq \int_{G} |h(y^{-1}x)| w(y^{-1}x) d\mu(x) \\
            & = \int_{G} |h(u)| w(u)  du =\|h\|_{L^{1}_{w}(G)}.
       \end{split}
   \end{equation*}
   Now the claim follows immediately since the assumptions of Lemma \ref{weightedorliczschur} are satisfied.
\end{proof}
In the following, we show that the pointwise evaluation on $V,$ the image space of an idempotent integral operator on $L^{\phi}_{1/w}(G),$ is bounded.
\begin{prop}
  Let $V$ be the image space of an idempotent integral operator $T$ on $L^{\phi}_{1/w}(G)$ with the integral kernel $K$ satisfying Schur's condition. If there is an $r>0$ such that $\|\operatorname{osc}_{r}(K)\|_{S} < \infty,$ then pointwise evaluation is bounded on $V.$
\end{prop}
\begin{proof}
  Let $r>0$ be such that the hypothesis is satisfied. Now, for $x, y' \in G,$ consider
\begin{equation*}
    \begin{split}
        |K(x,y')| \leq \operatorname{osc}_{r}(K)(x,y) + |K(x,y)|  \quad \mbox{ for } y \in B(y',r).
    \end{split}
\end{equation*}
 This implies, 
\begin{equation*}
    \begin{split}
    |K(x,y')| \frac{w(y')}{w(x)} &\leq \operatorname{osc}_{r}(K)(x,y)\frac{w(y')}{w(x)} + |K(x,y)|\frac{w(y')}{w(x)}   \\
    & \leq M_{\infty} \left(\operatorname{osc}_{r}(K)(x,y)\frac{w(y)}{w(x)} + |K(x,y)|\frac{w(y)}{w(x)}\right) 
    \end{split}
\end{equation*}
$\mbox{for } y \in B(y',r).$ Now integrate with respect to the variable $y$ on $B(y',r)$
\begin{equation*}
 \begin{split}
    |K(x,y')| \frac{w(y')}{w(x)} \mu(B(y',r)) &\leq  M_{\infty} \int_{B(y',r)}\left(\operatorname{osc}_{r}(K)(x,y)\frac{w(y)}{w(x)} + |K(x,y)|\frac{w(y)}{w(x)}\right) d\mu(y) \\
    &\leq M_{\infty} \int_{G}\left(\operatorname{osc}_{r}(K)(x,y)\frac{w(y)}{w(x)} + |K(x,y)|\frac{w(y)}{w(x)}\right) d\mu(y).
 \end{split}
\end{equation*}
Therefore 
\begin{align}
    \|K(x,\cdot) w(\cdot)\|_{L^\infty (G)} &\leq \sup_{y' \in G} \frac{1}{\mu(B(y',r))}\left( \|\operatorname{osc}_{r}(K)\|_{S} + \|K\|_{S} \right) w(x) \\
    &= \frac{w(x)}{\mu(B(e,r))}\left( \|\operatorname{osc}_{r}(K)\|_{S} + \|K\|_{S} \right).  
\end{align}
Also, it is obvious that 
\begin{equation}
    \|K(x,\cdot) w(\cdot)\|_{L^{1}(G)} \leq w(x) \|K\|_{S}.
\end{equation}
Thus, for each $x \in G,$ $K(x,\cdot) w(\cdot)$ belongs to the space $L^{1}(G)\cap L^{\infty}(G),$ and therefore belongs to $L^{\phi^*} (G),$ see, for example, \cite{BennettSharpleyInterpolation} for such a result. Hence, by applying H\"older's inequality to \eqref{orliczintegralop}, we get that the pointwise evaluation is bounded.
\end{proof} 
Next is Minkowski's integral inequality for the Orlicz norm. We cannot find a proof for this in the literature, so for completeness, we include one here.
\begin{lemma}[Minkowski's integral inequality for Orlicz norm]\label{minkowskiorlicz}
Let $f$ be a non-negative measurable function defined on $G \times G.$ Then the following holds
\begin{equation}\label{minkowskiineqorlicz}
\displaystyle \left\|\int_{G} f(x,\cdot) ~ d\mu(x)\right \|_{L^{\phi}(G)} \leq \int_{G} \|f(x,\cdot)\|_{L^{\phi}(G)} ~ d\mu(x).    
\end{equation}

\end{lemma}
\begin{proof}
    Let us denote the integrand on the RHS of \eqref{minkowskiineqorlicz} by $h(x)=\|f(x,\cdot)\|_{L^{\phi}(G)}$ and its $L^{1}$ norm by $\displaystyle H=\int_{G} h(x) ~ d\mu(x)$ in order to avoid the cumbersome notation. Note that the results hold trivially if $H=0$ or $\infty.$ Thus, for the case when $0<H < \infty,$ it is enough to show 
    $$ \displaystyle \int_{G} \phi \left( \frac{\int_{G}|f(x,y)| ~ d\mu(x)}{H}\right) ~ d\mu(y) \leq 1.$$
    Let us consider the following
    \begin{equation*}
        \begin{split}
             \phi \left(\int_{G} \frac{|f(x,y)|}{H} ~ d\mu(x) \right) &= \phi \left(\int_{h \neq 0} \frac{|f(x,y)|}{h(x)} \frac{h(x)}{H} ~ d\mu(x) \right) \\
             & \leq \int_{h \neq 0} \phi \left( \frac{|f(x,y)|}{h(x)} \right) \frac{h(x)}{H} ~ d\mu(x)
        \end{split}
    \end{equation*}
    where we have used Jensen's in the last inequality. Therefore
    \begin{equation*}
        \begin{split}
            \int_{G} \phi \left( \frac{\int_{G}|f(x,y)| ~ d\mu(x)}{H}\right) d \mu (y) &\leq \int_{G }\int_{h \neq 0} \phi \left( \frac{|f(x,y)|}{h(x)} \right) \frac{h(x)}{H} ~ d\mu(x) ~ d\mu(y) \\
           \mbox{(Fubini's)} &\leq \int_{h \neq 0 }\int_{G} \phi \left( \frac{|f(x,y)|}{h(x)} \right) \frac{h(x)}{H} ~ d\mu(y) ~ d\mu(x) \\
           & = \int_{h \neq 0} \left( \int_{G} \phi \left( \frac{|f(x,y)|}{h(x)} \right) ~d\mu(y)\right) \frac{h(x)}{H} ~ d\mu(x) \\
           &\leq \int_{G}  \frac{h(x)}{H} ~d\mu(x) \leq 1. 
        \end{split}
    \end{equation*}
    This proves the claim.
\end{proof} 
In the following, we discuss a boundedness result for the integral operators on the mixed-norm Orlicz spaces $\lphps_{1/w}(G \times G).$ For a similar result on the mixed Lebesgue spaces see \cite{FelixNickiSchurMIxedLebSP}. \par
To streamline the notation for the next lemma, define
$$U(b,y)=\sup_x\int_G |K((a,b),(x,y))| \frac{w(x,y)}{w(a,b)}\,d\mu(a),  U_b = \int_G U(b,y)\,d\mu(y), ~ U^y = \int_G U(b,y)\,d\mu(b),$$
and
$$V(b,y)=\sup_a\int_G |K((a,b),(x,y))| \frac{w(x,y)}{w(a,b)}\,d\mu(x), V_b = \int_G V(b,y)\,d\mu(y), V^y = \int_G V(b,y)\,d\mu(b).$$

\begin{lemma}[Weighted Schur's Lemma for mixed norm Orlicz Spaces]\label{Schurmixedorlicz}
  Let $K$ be a Borel measurable function on $G \times G \times G \times G,$ and $w$ be a weight function on $G \times G.$ Consider the following integral 
    \begin{align}
        Tf(a,b)= \int_{G} \int_{G} K((a,b),(x,y)) f(x,y) ~ d\mu(x) ~d\mu(y).
    \end{align}
    If the integral kernel $K$ satisfies 
\begin{equation*}
     \|K\|_{CS}=\max\left\{\sup_b U_b,\;\sup_y U^y,
\;\sup_b V_b,\;\sup_y V^y\right\} < \infty
    \end{equation*}
    which we call as the Cross-Schur condition, then $T$ is a bounded operator on $L^{\phi,\psi}_{1/w}(G \times G)$ and satisfies the following bound 
    \begin{align*}
    \|Tf\|_{L^{\phi,\psi}_{1/w}(G \times G)} \leq 2 \|K\|_{CS} \|f\|_{L^{\phi,\psi}_{1/w}(G \times G)} ,\quad \forall f \in L^{\phi,\psi}_{1/w}(G \times G).
    \end{align*}
\end{lemma}
\begin{proof}
    The estimate
\begin{align*}
\|Tf(\cdot,b) w^{-1}(\cdot, b)\|_{L^{\phi}(G)}\leq\int_G\left\|\int_G |K((a,b),(x,y))| \frac{w(x,y)}{w(a,b)}\,\frac{|f(x,y)|}{w(x,y)}\,d\mu(x)\right\|_{L^{\phi}(G)}\,d\mu(y)
\end{align*}
follows from Minkowski's integral inequality \eqref{minkowskiineqorlicz} for the Orlicz norm. For fixed $b$ and $y$, define
$$K_{(b,y)}(a,x)=K((a,b),(x,y)),\qquad f_y(x)=f(x,y).$$
Then
$$\left\|\int_G |K((a,b),(x,y))| \frac{w(x,y)}{w(a,b)}\,\frac{|f(x,y)|}{w(x,y)}\,d\mu(x)\right\|_{L^{\phi}(G)}=\left\|
\int_G |K_{(b,y)}(a,x)|\frac{w(x,y)}{w(a,b)} \, \frac{|f_y(x)|}{w(x,y)}\,d\mu(x)\right\|_{L^{\phi}(G)}.$$
By weighted Schur's Lemma \ref{weightedorliczschur}, if
$$\max\left\{\sup_x \int_G |K_{(b,y)}(a,x)|\frac{w(x,y)}{w(a,b)}\,d\mu(a),\;\sup_a \int_G |K_{(b,y)}(a,x)|\frac{w(x,y)}{w(a,b)}\,d\mu(x)\right\}<\infty,$$
then
$$\left\|\int_G \frac{|K_{(b,y)}(a,x)|}{w(a,b)}\,|f_y(x)| \,d\mu(x)\right\|_{L^{\phi}(G)}\leq
\|K_{(b,y)}\|_{S}\, \left\| f(\cdot,y)w^{-1}(\cdot,y)\right\|_{L^{\phi}(G)}.$$
Consequently,
\begin{align*}
\|Tf(\cdot,b)w^{-1}(\cdot,b)\|_{L^{\phi}(G)}\leq\int_G\|K_{(b,y)}\|_{S}\,\left\|f(\cdot,y)w^{-1}(\cdot,y)\right\|_{L^{\phi}(G)}\,d\mu(y).
\end{align*}
Taking the $L^{\psi}$-norm with respect to $b$, we obtain
\begin{align*}
\left\|\|Tf(\cdot,b) w^{-1}(\cdot,b)\|_{L^{\phi}(G)}\right\|_{L^{\psi}(G)}\leq\left\|\int_G
\|K_{(b,y)}\|_{S}\,\|f(\cdot,y) w^{-1}(\cdot,y)\|_{L^{\phi}(G)}\,d\mu(y)\right\|_{L^{\psi}(G)}.
\end{align*}
The right-hand side of the inequality above sets itself up for a second application of Schur's lemma. 
By the assumption $\|K\|_{CS}<\infty$, and it therefore allows for a second application of Schur's lemma
$$\|Tf\|_{L^{\phi,\psi}_{1/w}(G\times G)}\leq 2\|K\|_{CS}\,\|f\|_{L^{\phi,\psi}_{1/w}(G \times G)}.$$

Finally, observe that the condition $\|K\|_{CS}<\infty$ implies that $U(b,y)$ and $V(b,y)$ are finite for almost every $b$ and $y$ in $G$. Therefore, the Schur conditions for the sliced kernel $K_{(b,y)}$ are automatically satisfied for almost every $b$ and $y$. So the first application of weighted Schur's lemma is justified.
\end{proof}
\begin{remark}
    Using the above result, one can get generalized Young's type inequality for the mixed-norm Orlicz spaces.
\end{remark}
The next result establishes the boundedness of pointwise evaluation on the image space $V$ of $L^{\phi,\psi}_{1/w}(G \times G)$ under the integral operator. In order to ensure this, we need some additional conditions on the integral kernel $K,$ which we list below for easy reference.\\
\textbf{Assumption 1 (A1):} $\displaystyle  \int_{G} \sup_{a,b,y \in G} |K((a,b),(x,y))| w(x,y) d \mu(x) < \infty.$  \\
\textbf{Assumption 2 (A2):} $\displaystyle  \int_{G}  \sup_{a,b,x \in G} |K((a,b),(x,y))| w(x,y) d \mu(y) < \infty.$
\begin{prop}
  Let $V$ be the image space of an idempotent integral operator $T$ on $L^{\phi,\psi}_{1/w}(G \times G)$ with the integral kernel $K.$ Assume that $K$ satisfies the Cross-Schur condition, (A1), and (A2). If there is an $r > 0$ such that $\|\operatorname{osc}_{r}(K)\|_{CS} < \infty,$ then the pointwise evaluation on $V$ is bounded. 
\end{prop}
\begin{proof}
    It is enough to show that $\|K((a,b),(\cdot,\cdot)) w(\cdot,\cdot)\|_{L^{\phi^*, \psi^*}(G \times G)} < \infty,$ then the results follows from the application of H\"older's inequality. To this end, observe
    \begin{multline*}
         \left\| K((a,b),(\cdot,\cdot)) w(\cdot,\cdot) \right\|_{L^{\phi^*, \psi^*}(G \times G)} 
    \leq \left\| y \mapsto \| K((a,b),(\cdot,y)) w(\cdot,y) \|_{L^{1}(G)} \|\right\|_{L^{1}(G)} + \\
      \left\| y \mapsto\| K((a,b),(\cdot,y)) w(\cdot,y) \|_{L^{1}(G)}\right\|_{L^{\infty}(G)} +   
      \left\| y \mapsto \| K((a,b),(\cdot,y)) w(\cdot,y) \|_{L^{\infty}(G)}  \right\|_{L^{1}(G)} + \\
      \left\| y \mapsto\| K((a,b),(\cdot,y)) w(\cdot,y) \|_{L^{\infty}(G)}\right\|_{L^{\infty}(G)}.
    \end{multline*}
    By the assumptions, all the norms appearing on the right-hand side are finite, which completes the proof.
\end{proof}

\section{Sampling in a subspace of Orlicz space} \label{samplingorliczsection}
We are going to consider the sampling of non-decaying signals which belong to the image space of an idempotent integral operator on $L^{\phi}(G)$.\par
Consider the integral operator $T$ 
\begin{equation}\label{orliczintegralop}
Tf(x)=\int_{G} K(x,y) f(y) ~ d\mu(y)    \quad \mbox{ for } f \in \LLP_{1/w}(G). 
\end{equation}

The integral operator $T$ is considered to be idempotent $(T^2=T)$, and the weight $w$ is assumed to be sub-multiplicative. We also assume that the integral kernel $K$ satisfies the Schur condition $\|K\|_{S} < \infty.$ This makes the integral operator to be a well-defined and bounded operator on $\LLP_{1/w}(G).$ \par
Let us denote, by $V,$ the image space of the idempotent integral operator $T$
\begin{equation*}
    V:=\{Tf: f \in \LLP_{1/w}(G)\}.
\end{equation*}
Note that, since $T$ is idempotent, we have 
\begin{equation*}
    Tf=f,  \quad\forall f \in V.
\end{equation*}

For a relatively separated and $r-$dense set $X=\{x_{i} : i \in I\}$, we define the following sampling operator  
$$\displaystyle \mathcal{S}f(x):= \sum_{ i\in I} f(x_{i}) \frac{\chi_{B(x_i,r)}(x)}{N_{X}(x)}.$$
It can be noted that this operator is well-defined pointwise owing to the fact that the set $X$ is relatively separated. The bounding estimates of the operator $\mathcal{S}$ will lead to a sampling inequality for the space $V.$\par
The following theorem establishes the existence of a stable set of sampling for the space $V.$
\begin{theorem}\label{samplingorlicz}
    Let $T$ be an idempotent integral operator on $L^{\phi}_{1/w}(G)$ with integral kernel $K$ satisfying Schur's condition. Further, let $X=\{x_i: i \in I\}$ be a relatively separated and $r-$dense set in $G.$ Suppose there exists an $r>0$ such that 
    $$\|\operatorname{osc}_{r}(K)\|_S < 1,$$
    then any signal $f \in V$ can be stably reconstructed from its sample values on the set $X.$ Moreover, the sampling inequality 
    \begin{equation}
       \begin{split}
        \frac{1}{L_X}\left( 1-\|\operatorname{osc}_{r}(K)\|_S \right) \|f\|_{\LLP_{1/w}(G)} \lesssim  \|(f(x_i))_{i \in I}\|_{\ell^{\phi}_{1/w}(X)} \lesssim L_{X}^2 \|f\|_{\LLP_{1/w}(G)}\left(  1+ \|\operatorname{osc}_{r}(K)\|_S\right) 
        \end{split}
    \end{equation}
    holds for all $f \in V.$
\end{theorem}
\begin{proof}
    With all the assumptions in place, we estimate the operator $S$ applied on signals $f \in V.$ We begin by considering 
    \begin{equation*}
    \begin{split}
        |f(x)-\mathcal{S}f(x)| & =  \left|f(x)-\sum_{i \in I} f(x_{i}) \frac{\chi_{B(x_i,r)}(x)}{N_{X}(x)}\right| \\
        &= \left|f(x) \sum_{i \in I} \frac{\chi_{B(x_i,r)}(x)}{N_{X}(x)}-\sum_{i \in I} f(x_{i}) \frac{\chi_{B(x_i,r)}(x)}{N_{X}(x)} \right|  \\
        & \leq \sum_{i \in I} |f(x)-f(x_i)| \frac{\chi_{B(x_i,r)}(x)}{N_{X}(x)}. 
    \end{split}
\end{equation*}
Now, for $x \in B(x_i,r)$
\begin{equation*}
    |f(x)-f(x_i)| \leq \operatorname{osc}_{r}(f)(x)
\end{equation*}
and therefore
\begin{equation*}
    \begin{split}
        |f(x)-\mathcal{S}f(x)| & \leq \sum_{i \in I} \operatorname{osc}_{r}(f)(x)  \frac{\chi_{B(x_i,r)}(x)}{N_{X}(x)}  \\
        &= \operatorname{osc}_{r}(f)(x).
    \end{split}
\end{equation*}

For $\lambda >0,$ and by the monotonicity of $\phi$ we have 
\begin{equation*}
    \begin{split}
      \phi \left( \frac{ |f(x)-\mathcal{S}f(x)|}{w(x)\lambda} \right) & \leq \phi \left( \frac{ \operatorname{osc}_{r}({f})(x)} {w(x)\lambda} \right).
    \end{split}
\end{equation*}
Integrating against the variable $x$ over $G,$ we get the following
\begin{equation}\label{diffinnorm}
    \|f- \mathcal{S}f\|_{\LLP_{1/w}(G)} \leq \|\operatorname{osc}_{r}({f})\|_{\LLP_{1/w}(G)}.
    \end{equation}
    In what follows, we reduce the above estimate from the oscillation of the function $f$ to that of the kernel $K$. To do so, we write 
    \begin{equation*}
        \begin{split}
            \operatorname{osc}_{r}({f})(x) &= \sup_{d(y,x)<r} |f(x)-f(y)| \\
            &= \sup_{d(y,x)<r} |Tf(x)-Tf(y)| \\
            &= \sup_{d(y,x)<r } \left | \int_{G}K(y,t)f(t)~ dt - \int_{G} K(x,t) f(t)~ dt \right | \\
            & \leq \sup_{d(y,x)<r} \int_{G} \left|K(y,t)-K(x,t) \right| |f(t)| ~ dt \\
            & \leq \int_{G} \sup_{d(y,x)<r, d(z,t)<r} \left|K(y,z)-K(x,t) \right| |f(t)| ~ dt \\
            &= \int_{G} \operatorname{osc}_{r}({K})(x,t) |f(t)| ~ dt.
        \end{split}
    \end{equation*}
    By Schur's lemma
    \begin{equation*}\label{oscftoosck}
        \|\operatorname{osc}_{r} ({f})\|_{\LLP_{1/w}(G)} \leq \|f\|_{\LLP_{1/w}(G)} \|\operatorname{osc}_{r}({K})\|_{S}.
    \end{equation*}
   Putting this in \eqref{diffinnorm}, we get 
   \begin{equation*}\label{diffinnorm1}
       \|f- \mathcal{S}f\|_{\LLP_{1/w}(G)} \leq \|f\|_{\LLP_{1/w}(G)} \|\operatorname{osc}_r(K)\|_{S}.
   \end{equation*}
    We now turn to estimating $\|\mathcal{S}f\|_{\LLP_{1/w}(G)}.$ But with the estimates proved in Lemma \ref{seqequiv}, this is in fact straightforward, and we have
    \begin{equation*}
        \|\mathcal{S}f\|_{\LLP_{1/w}(G)} \leq \left\|\sum_{i \in I} |f(x_i)| \chi_{B(x_i,r)}\right\|_{\LLP_{1/w}(G)} \lesssim L_{X}\|(f(x_i))_{i \in I}\|_{\ell^{\phi}_{1/w}(X)}.
    \end{equation*}
    To get the lower sampling bound, one now observes that
    \begin{equation*}
        \begin{split}
            \|\mathcal{S}f\|_{\LLP_{1/w}(G)} &\geq \|f\|_{\LLP_{1/w}(G)} -\|\mathcal{S}f-f\|_{\LLP_{1/w}(G)} \\
            &\geq \left(1-\|\operatorname{osc}_{r}(K)\|_{S} \right) \|f\|_{\LLP_{1/w}(G)}.
        \end{split}
    \end{equation*}
    Combining, we get 
    \begin{equation*}
        \frac{1}{L_X}\left( 1-\|\operatorname{osc}_{r}(K)\|_S \right) \|f\|_{\LLP_{1/w}(G)} \lesssim \|(f(x_i))_{i \in I}\|_{\ell^{\phi}_{1/w}(X)}. 
    \end{equation*}
    Now, for $x \in B(x_i,r)$
    $$|f(x_i)| \leq |f(x)| + \operatorname{osc}_{r}(f)(x)$$
    which gives
    $$\sum_{i \in I} |f(x_i)| \frac{\chi_{B(x_i,r)}(x)}{N_{X}(x)} \leq  |f(x)|+ \operatorname{osc}_{r}(f)(x)$$
    and therefore 
    \begin{equation*}
        \begin{split}
           \left \|\sum_{i \in I} |f(x_i)| \frac{\chi_{B(x_i,r)}}{N_{X}}\right\|_{\LLP_{1/w}(G)} &\leq \|f\|_{\LLP_{1/w}(G)} + \|\operatorname{osc}_{r}(f)\|_{\LLP_{1/w}(G)} \\ 
            &\leq \|f\|_{\LLP_{1/w}(G)} \left ( 1+ \|\operatorname{osc}_{r}(K)\|_{S} \right).
        \end{split}
    \end{equation*}
    Again by \eqref{seqequiveqn}, 
    \begin{equation*}
        \begin{split}
            \left \|\sum_{i \in I} |f(x_i)| \frac{\chi_{B(x_i,r)}}{N_{X}}\right\|_{\LLP_{1/w}(G)} &\geq \frac{1}{L_{X}} \left\|\sum_{i \in I} |f(x_i)| \chi_{B(x_i,r)}\right\|_{\LLP_{1/w}(G)}  \\
            & \gtrsim \frac{1}{L_{X}^2} \|(f(x_i))_{i \in I}\|_{\ell^{\phi}_{1/w}(X)}.
        \end{split}
    \end{equation*}
   Hence, we have  
   \begin{equation*}
        \frac{1}{L_X}\left( 1-\|\operatorname{osc}_{r}(K)\|_S \right) \|f\|_{\LLP_{1/w}(G)} \lesssim  \|(f(x_i))_{i \in I}\|_{\ell^{\phi}_{1/w}(X)} \lesssim L_{X}^2 \|f\|_{\LLP_{1/w}(G)}\left(  1+ \|\operatorname{osc}_{r}(K)\|_S\right)
   \end{equation*}
   which concludes the proof.
\end{proof}

In many practical applications, it is not feasible to measure the instantaneous values of a signal. Instead, one records its average values over suitable spatial regions. Motivated by this, we consider the problem of average sampling of non-decaying signals within the same framework. The average sampling theorem reads as follows.
\begin{theorem}\label{avgsamplingorlicz}
    Let $T$ be an idempotent integral operator on $L^{\phi}_{1/w}(G)$ with integral kernel $K$ satisfying Schur's condition. Further, let $X=\{x_i: i \in I\}$ be a relatively separated and $r-$dense set in $G.$ Assume there exists a family of averaging functions $\{\eta_i: i \in I\}$ satisfying 
    \begin{itemize}
        \item $0 \leq \eta_i \leq 1$ for all $i \in I,$
        \item $\displaystyle\int_{G} \eta_i =1$ for all $i \in I,$ and
        \item for all $i \in I,$ $\operatorname{supp}(\eta_i) \subset B(x_i,r).$
    \end{itemize}
    If there is a $r>0$ such that 
    \begin{equation}
        \|\operatorname{osc}_{2r} (K)\|_{S} < 1,
    \end{equation}
    then any signal $f \in V$ can be recovered from its averages $\displaystyle \left\{\langle f,\eta_i \rangle=\int_{G} f(t) \eta_{i}(t) ~dt: i \in I\right\}.$ Moreover, the following average sampling inequality 
    \begin{equation}
    \frac{1}{L_X} (1-\|\operatorname{osc}_{2r}(K)\|_S) \|f\|_{\LLP_{1/w}(G)} \lesssim \|(\langle f,\eta_i \rangle)_{i \in I}\|_{\ell^{\phi}_{1/w}(X)} \lesssim  L_{X}^2 \|f\|_{\LLP_{1/w}(G)}\left(  1+ \|\operatorname{osc}_{2r}(K)\|_S\right)
        \end{equation}
        holds for all $f \in V.$
\end{theorem}
\begin{proof}
    Analogous to the case of pointwise sampling, we define the average sampling operator as 
    $$\mathcal{S}_{avg}f(x):= \sum_{i \in I} \langle f, \eta_{i} \rangle \frac{\chi_{B(x_i,r)}(x)}{N_{X}(x)},$$
   for $f \in V.$ \\
  The result follows by repeating the arguments in the proof of Theorem \ref{samplingorlicz}, with only minor modifications to accommodate the present setting.    
\end{proof}

\subsection{Random Sampling of signals in Orlicz space}
In this subsection, we will consider the sampling for the functions whose norm content is essentially located on a compact set. In particular, we consider the following: \\
For $\delta >0,$ define the set of essentially concentrated functions as follows 
$$V_{N,\delta}=\{f \in V: \|f\|_{\LLP_{1/w}(G_N)} \geq (1-\delta) \|f\|_{\LLP_{1/w}(G)} \},$$
where $G_N$ is the closure of the ball of radius $N$ centered at the identity $e.$
\begin{definition}[Covering radius]
    Let $X=\{x_1, x_2, \ldots, x_n\}$ be a subset of $G_{N}.$ The covering radius of $X$ for $G_{N}$ is defined as the following
$$r_{X}(n):= \inf\{r >0: G_{N} \subset \bigcup_{j=1}^{n} B(x_{j},r)\}.$$
\end{definition}
 Note that $G_N$ being a compact set ensures that the covering radius is finite.
Since the functions in $V_{N,\delta}$ have their norm essentially concentrated on a compact subset, it is natural to expect that they can be stably recovered from their samples taken on that subset. The following result makes this heuristic precise. 
\begin{lemma} \label{Orliczcoveringradiussampling}
    Let $T$ be an idempotent integral operator on $L^{\phi}_{1/w}(G)$ with integral kernel $K$ satisfying Schur's condition. Suppose there exists an $r>0$ such that 
    \begin{equation}
        \|\operatorname{osc}_{r}(K)\|_S < 1-\delta ,
    \end{equation}
    and let $X=\{x_i: i=1, 2, \ldots,n\}$ be a finite subset of distinct points in $G_N$ such that $\{B(x_i,r): i=1,2, \ldots, n\}$ covers $G_N.$ Then, for all $f \in V_{N,\delta}$ the following inequality holds 
    \begin{multline}
        (1-\delta-\|\operatorname{osc}_{r}(K)\|_S)\|f\|_{\LLP_{1/w}(G)} \leq \|(f(x_i))_{i=1}^{n}\|_{\ell^{\phi}_{1/w}(X)} \\ \leq \frac{1}{\mu(B(e,\zeta))} \|f\|_{\LLP_{1/w}(G)} (1+ \|\operatorname{osc}_{r}(K)\|_{S}),
    \end{multline}
   where $\zeta$ depends on the separating distance between the sampling points.
\end{lemma}
\begin{proof}
    Let $r>0$ be such that the assumption above on the kernel $K$ is satisfied. Take $X=\{x_1,x_2,\ldots, x_n\}$ to be a set of points from $G_N$ such that it is covered by $\{B(x_i,r): i=1,2,\ldots,n\}.$ From this collection, we can extract a collection $\{B_i: i=1,2,\ldots,n\}$ such that
    \begin{enumerate}
        \item there is an $\zeta > 0$ such that $B(x_i', \zeta) \subset B_i \subset B(x_i,r) \cap G_N,$ and $x_i \in B_i$ for all $i=1,2,\ldots,n,$
        \item $B_i \cap B_j =\varnothing $ for all $i \neq j,$
        \item $G_N = \bigcup_{i=1}^{n} B_{i}.$
    \end{enumerate}
    Now, define the sampling operator as
    \begin{equation*}
        \mathcal{A}f(x):= \sum_{i=1}^{n} f(x_i) \chi_{B_i}(x)
    \end{equation*}
    for $x \in G_N$ and $f \in V_{N,\delta}.$ \\
    By Lemma \ref{seqequiv}, 
    \begin{equation*}
    \|\mathcal{A}f\|_{L^{\phi}_{1/w}(G_{N})} \leq \|(f(x_i))_{i=1}^{n}\|_{\ell^{\phi}_{1/w}(X)}.
\end{equation*}
Also, as previously deduced, one can see that 
\begin{equation*}
    \|\mathcal{A}f-f\|_{\LLP_{1/w}(G_N)} \leq \|\operatorname{osc}_{r}(f)\|_{\LLP_{1/w}(G_N)} \leq \|f\|_{\LLP_{1/w}(G)} \|\operatorname{osc}_{r}(K)\|_{S}.
\end{equation*}
Thus, 
\begin{equation*}
    \begin{split}
     \|(f(x_i))_{i=1}^{n}\|_{\ell^{\phi}_{1/w}(X)} &\geq  \|\mathcal{A}f\|_{L^{\phi}_{1/w}(G_{N})} \\
    & \geq \|f\|_{\LLP_{1/w}(G_N)} - \|\mathcal{A}f-f\|_{\LLP_{1/w}(G_N)} \\
    &\geq (1-\delta)\|f\|_{\LLP_{1/w}(G)}- \|f\|_{\LLP_{1/w}(G)}\|\operatorname{osc}_{r}(K)\|_{S} \\
    &=(1-\delta-\|\operatorname{osc}_{r}(K)\|_{S}) \|f\|_{\LLP_{1/w}(G)}, \quad \forall f \in V_{N,\delta}.
    \end{split}
\end{equation*}
Now for $x \in G_N$
\begin{equation*}
    \sum_{i=1}^{n} |f(x_i)| \chi_{B_{i}}(x) \leq |f(x)|+ \operatorname{osc}_{r}(f)(x).
\end{equation*}
Therefore, 
\begin{align*}
    \mu(B(e,\zeta)) \|(f(x_i))_{i=1}^{n}\|_{\ell^{\phi}_{1/w}(X)} &\leq \left\| \sum_{i=1}^{n} |f(x_i)| \chi_{B_i}\right\|_{\LLP_{1/w}(G)} \\
    &\leq \|f\|_{\LLP_{1/w}(G)} (1+ \|\operatorname{osc}_{r}(K)\|_S) \\
    &\leq (2-\delta) \|f\|_{\LLP_{1/w}(G)},
\end{align*}
which concludes the proof.
\end{proof} 
To prove the random sampling result, we assume an additional regularity condition on the integral kernel $K.$ In particular, we assume 
\begin{equation}
\label{kernelregularity}
    \kappa_{\theta}= \|K\|_{S} + \sup_{0<r\leq 1}r^{-\theta} \|\operatorname{osc}_{r}(K)\|_{S} < \infty
\end{equation}
for some $\theta \in (0,1].$ Such an assumption is standard in the theory of random sampling; see \cite{QSunRandomACHA}.
\begin{remark} \label{remarkregradius}
    If, in the previous lemma, we additionally assume that the kernel satisfies the regularity condition (\ref{kernelregularity}), then any finite set of points
$$X=\{x_1,x_2,\ldots,x_n\}$$
with covering radius
\begin{equation}
\label{regularityradiusrelation}
    r_X(n)<\left(\frac{1-\delta}{\kappa_\theta}\right)^{1/\theta}
\end{equation}
yields the following sampling inequality 
\begin{multline}
    (1-\delta-r_{X}(n)^{\theta} \kappa_{\theta})\|f\|_{\LLP_{1/w}(G)} \leq \|(f(x_i))_{i=1}^{n}\|_{\ell^{\phi}_{1/w}(X)} \\ \leq \frac{1}{B(e,\zeta)} \|f\|_{\LLP_{1/w}(G)} (1+ r_{X}(n)^{\theta} \kappa_{\theta})
\end{multline}
for all $f \in V_{N,\delta},$ and $\zeta$ depending on the separating distance between the sampling points . This essentially says that if the covering radius is small enough, then one gets a sampling inequality for the $\delta-$concentrated functions on $G_N.$ 
\end{remark}
\begin{remark}
For functions whose norm is essentially concentrated on a compact subset, the upper sampling inequality follows readily from the boundedness of the pointwise evaluation functionals. Thus, the main difficulty lies in establishing the corresponding lower sampling inequality, which is essential for stable recovery from the samples. The following theorem establishes a random sampling inequality by controlling the covering radius of the randomly chosen sampling points. In particular, the lower sampling inequality is obtained from this covering-radius estimate, while the upper sampling inequality follows directly from the boundedness of the pointwise evaluation functionals. 
\end{remark}

 \begin{theorem}\label{randomsamplingorlicz}
Let $T$ be an idempotent integral operator with integral kernel $K$ satisfying the condition (\ref{kernelregularity}). Assume $X=\{x_1,x_2,\ldots,x_n\}$ are i.i.d. random variables distributed uniformly over the set $G_N$ with respect to the probability measure $\displaystyle\frac{\chi_{G_N}(x)}{\mu(G_N)} d\mu(x)$. Then for $0 <\epsilon< 1-\delta,$ 
\begin{equation}
   (1-\delta-\epsilon)\|f\|_{\LLP_{1/w}(G)} \leq \|(f(x_i))_{i=1}^{n}\|_{\ell^{\phi}_{1/w}(X)} \leq c_n D_N ||f\|_{\LLP_{1/w}(G)}
\end{equation}
holds for all $f \in V_{N,\delta}$ with probability atleast 
\begin{equation}
    1- \frac{\mu(G_N)}{\mu(B(e,\epsilon^{1/\theta} \kappa^{-1/\theta}/4))} \exp\left(-n\frac{\mu(B(e,\epsilon^{1/\theta} \kappa^{-1/\theta}/2))}{\mu(G_N)}\right),
\end{equation}
where  $D_N = \sup_{x \in G_N} \|K(x,\cdot)\|_{L^{\phi^*}_{w}(G)}$ and $c_n$ is a constant depending on the number of points $n$.
 \end{theorem}
\begin{proof}
Let $X=\{x_1,x_2,\ldots,x_n\}$ be a collection of i.i.d. random variables uniformly distributed over the set $G_N$, $\alpha > 0$ and define the event $E=\{r_{X}(n) \geq \alpha\}.$ By Remark \ref{remarkregradius}, if the covering radius of $X$ satisfies
$$r_X(n)<\left(\frac{1-\delta}{\kappa_\theta}\right)^{1/\theta},$$
then $X$ is a stable sampling set for $V_{N,\delta}$.

The proof is motivated from \cite{GrochenigBassRandomMultiTrig}. Let $\Gamma \subset G_N$ be an $\alpha/2$-net of $G_N$, whose cardinality does not exceed $\frac{\mu(G_N)}{\mu(B(e,\alpha/4))}.$ The existence of such a cover is ensured by a standard maximal packing argument. Then the collection
$$\mathcal{C} = \{B(\gamma,\alpha/2): \gamma \in \Gamma\}$$
forms a cover of $G_N$.

Assume that $E$ is true for some $ 0<\alpha <1 .$ Since $r_X(n)\geq\alpha$, there exists a point $x \in G_N$ whose distance from every sample point is at least $\alpha$ or exceeds it. Let $\gamma \in \Gamma$ be such that $x \in B(\gamma,\alpha/2)$. Then
$$B(\gamma,\alpha/2)\cap X=\varnothing,$$
for otherwise there would exist $x_i \in X$ with
$$d(x,x_i)\leq d(x,\gamma)+d(\gamma,x_i)<\alpha,$$
contradicting the assumption that $r_X(n)\geq\alpha$. Hence
$$E\subset\left\{\exists\,\gamma\in\Gamma:B(\gamma,\alpha/2)\cap X=\varnothing\right\}.$$

Since the random variables are i.i.d. and uniformly distributed, we have
$$\mathbb{P}\left(\{B(\gamma,\alpha/2)\cap X=\varnothing\}\right)=\left(1-\frac{\mu(B(e,\alpha/2))}{\mu(G_N)}\right)^n.
    $$
The desired estimate now follows by applying the union bound 
\begin{align*}
    \mathbb{P}(E) &=\mathbb{P}(\{r_X(n) \geq \alpha\}) \\
    & \leq \mathbb{P}(\left\{\exists\,\gamma\in\Gamma:B(\gamma,\alpha/2)\cap X=\varnothing\right\}) \\
    &\leq \sum_{\gamma \in \Gamma} \mathbb{P}\left(\{B(\gamma,\alpha/2)\cap X=\varnothing\}\right) \\
    &\leq \frac{\mu(G_N)}{\mu(B(e,\alpha/4))}\left(1-\frac{\mu(B(e,\alpha/2))}{\mu(G_N)}\right)^n \\
    &\leq \frac{\mu(G_N)}{\mu(B(e,\alpha/4))} \exp\left(-n\frac{\mu(B(e,\alpha/2))}{\mu(G_N)}\right).
\end{align*}
By Remark \ref{remarkregradius}, the sampling inequality holds true if $r_{X}(n) < \left(\frac{1-\delta}{\kappa_{\theta}}\right)^{1/\theta}.$ In particular, for the choice $\alpha=\epsilon^{1/\theta} \kappa^{-1/\theta}$
the sampling inequality holds with probability at least
\begin{equation*}
    1- \frac{\mu(G_N)}{\mu(B(e,\epsilon^{1/\theta} \kappa^{-1/\theta}/4))} \exp\left(-n\frac{\mu(B(e,\epsilon^{1/\theta} \kappa^{-1/\theta}/2))}{\mu(G_N)}\right).
\end{equation*}
\end{proof} 
We would like to mention in passing that the constant $c(n) $ in the previous theorem can be made more explicit. In fact, one can take it to be $$c_n = \frac{2}{\sqrt{w(e)}}  \|\underbrace{(1,1,\cdots,1)}_{n-terms}\|_{\ell^{\phi}(X)}.$$
\begin{remark}
   The random sampling inequality \eqref{randomsamplingorlicz} holds with probability at least $1-\Upsilon$ whenever
$$n >\frac{\mu(G_N)}{\mu(B(e,\epsilon^{1/\theta}\kappa^{-1/\theta}/2))}\log\!\left(\frac{\mu(G_N)}{\Upsilon\,\mu(B(e,\epsilon^{1/\theta}\kappa^{-1/\theta}/4))}\right).$$
In particular, the required number of random samples is of order $\mathcal{O}\!\left(\mu(G_N)\log\mu(G_N)\right).$ This, in particular, coincides with the order of points known for the classical case \cite{PatelBajpeyiSivaRelevantOrlicz, GrochenigBassRandomMultiTrig, QSunRandomACHA}.
\end{remark}
\begin{remark}
   A more explicit bound for this probability can be given if the measure is known to be a \emph{Ahlfors b-regular}. A measure $\nu$ on a space $Z$ is said to be Ahlfors b-regular if there are positive constants $D_1, D_2$ such that 
   \begin{equation}
       D_1 r^b \leq \nu \left(\overline{B(z,r)}\right) \leq D_2 r^b \quad\mbox{ for all } z \in Z. 
   \end{equation}
   If $\mu$ is Ahlfors b-regular, then the probability that the sampling inequality holds is at least 
   \begin{equation*}
      1- \frac{D_2 N^b 4^b \kappa_{\theta}^{b/\theta}}{ D_1\epsilon^{b/\theta}}\exp{\left( -n\frac{D_1 \epsilon^{b/\theta}}{D_2 N^b 2^b\kappa_{\theta}^{b/\theta}} \right)}.
   \end{equation*}
\end{remark}
\begin{remark}
    Under the additional assumption of Ahlfors regularity, the compact ball can be replaced by a more general \emph{Corkscrew domain}. Consequently, the random sampling theorem continues to hold for signals whose norm is essentially concentrated on such domains. See \cite{QSunRandomACHA}.
\end{remark}

\section{Sampling in subspaces of mixed-norm Orlicz spaces}\label{samplingmixedorliczsection}
In this section, we extend the sampling results established for Orlicz spaces to the setting of mixed-norm Orlicz spaces. As in the previous section, we consider image spaces of an idempotent integral operator $V$, assuming some conditions, in addition to the Cross-Schur condition, on $K$ such that the pointwise evaluation is bounded on $V,$ and derive stable sampling inequalities under suitable assumptions on the kernel. Although the overall strategy remains the same, the mixed-norm setting requires additional care in handling the iterated Orlicz norms. 

\begin{theorem}\label{samplingmixedorlicz}
    Let $T$ be an idempotent integral operator on $L^{\phi,\psi}_{1/w}(G \times G)$ whose integral kernel $K$ satisfies the Cross-Schur condition, (A1) and (A2). Further, let $X=\{x_i: i \in I\}$ and $Y=\{y_j: j \in I\}$ be relatively separated and $r-$dense sets in $G.$ Suppose there exists an $r>0$ such that 
    $$\|\operatorname{osc}_{r}(K)\|_{CS} < 1/2,$$
    then any signal $f \in V$ can be stably reconstructed from its sample values on the set $X\times Y.$ Moreover, the sampling inequality 
    \begin{multline}
        \frac{1}{L_{X}  L_Y}(1-2\|\operatorname{osc}_{r}(K)\|_{CS}) \|f\|_{L^{\phi,\psi}_{1/w}(G \times G)}\lesssim \|(f(x_i,y_j))_{i,j \in I}\|_{\ell^{\phi,\psi}_{1/w}(X \times Y)} \\  \lesssim L_{X}^{2} L_{Y}^{3} (2\|\operatorname{osc}_{r}(K)\|_{CS}+1)\|f\|_{L^{\phi,\psi}_{1/w}(G \times G)}
    \end{multline}
    holds for all $f \in V.$
\end{theorem}
\begin{proof}
    As in the case of the Orlicz space, we define the sampling operator here 
    $$\displaystyle \mathcal{S}f(x,y):= \sum_{i,j}f(x_i,y_j) \frac{\chi_{B(x_i,r)} (x)}{N_X(x)} \frac{\chi_{B(y_j,r)}(y)}{N_Y(y)}.$$
Due to Lemma \ref{mixedseqequiv}, the proof reduces to the estimation of the norm of $\mathcal{S}f$ and the difference $\mathcal{S}f-f.$ The proof follows similarly to the case of Orlicz space. We write it out briefly. \\
From \eqref{mixedseqequiveqn}, one immediately has that 
$$\|\mathcal{S}f\|_{L^{\phi,\psi}_{1/w}(G \times G)} \lesssim L_X L_Y \|(f(x_i,y_j))_{i,j}\|_{\ell_{1/w}^{\phi,\psi}(X \times Y)}.$$
We also have the difference estimate
\begin{align*}
    |\mathcal{S}f(x,y)-f(x,y)| & = \left|\sum_{i,j}f(x_i,y_j) \frac{\chi_{B(x_i,r)}(x)}{N_{X}(x)} \frac{\chi_{B(y_j,r)}(y)}{N_{Y}(y)} -f(x,y) \sum_{i,j} \frac{\chi_{B(x_i,r)}(x)}{N_{X}(x)} \frac{\chi_{B(y_j,r)}(y)}{N_{Y}(y)}\right| \\
    &\leq \sum_{i,j}|f(x_i,y_j)-f(x,y)| \frac{\chi_{B(x_i,r)}(x)}{N_{X}(x)} \frac{\chi_{B(y_j,r)}(y)}{N_{Y}(y)} \\
    & \leq \operatorname{osc}_{r}(f)(x,y).
\end{align*}
This gives 
\begin{equation*}
    \|\mathcal{S}f-f\|_{L^{\phi,\psi}_{1/w}(G \times G)} \leq 2 \|\operatorname{osc}_{r}(K)\|_{CS} \|f\|_{L^{\phi,\psi}_{1/w}(G \times G)}
\end{equation*}
which leads to the lower sampling inequality 
\begin{equation*}
 \frac{1}{L_{X} L_Y}  (1-2\|\operatorname{osc}_{r}(K)\|_{CS})\|f\|_{L^{\phi,\psi}_{1/w}(G \times G)} \lesssim \|(f(x_i,y_j))_{i,j \in I}\|_{\ell^{\phi,\psi}_{1/w}(X \times Y)}.
\end{equation*}
Now, notice that
\begin{equation*}
    \sum_{i,j \in I} |f(x_i,y_j)| \frac{\chi_{B(x_i,r)}(x)}{N_{X}(x)}\frac{\chi_{B(y_j,r)}(y)}{N_{X}(y)} \leq |f(x,y)| + \operatorname{osc}_{r}(f)(x,y) 
\end{equation*}
which in turn gives 
\begin{equation*}
   \left\| \sum_{i,j \in I} |f(x_i,y_j)| \left(\frac{\chi_{B(x_i,r)}}{N_{X}}\otimes \frac{\chi_{B(y_j,r)}}{N_{Y}}\right)\right\|_{L^{\phi,\psi}_{1/w}(G \times G)} \leq \|f\|_{L^{\phi,\psi}_{1/w}(G \times G)}\left(1+2\|\operatorname{osc}_{r}(K)\|_{CS}\right).
\end{equation*}
By \eqref{mixedseqequiveqn}, we get the required upper sampling inequality 
\begin{equation*}
  \|(f(x_i,y_j))_{i,j \in I}\|_{\ell^{\phi,\psi}_{1/w}(X \times Y)} \lesssim L_X^2 L_Y^3 \left(1+2\|\operatorname{osc}_{r}(K)\|_{CS}\right)\|f\|_{L^{\phi,\psi}_{1/w}(G \times G)}. \qedhere
\end{equation*}
\end{proof} 
The next theorem addresses the average sampling in this setting.
\begin{theorem} \label{avgsamplingmixedorlicz}
Let $T$ be an idempotent integral operator on $L^{\phi,\psi}_{1/w}(G \times G)$ with integral kernel $K$ satisfying Cross-Schur condition, (A1), and (A2). Further, let $X=\{x_i: i \in I\}$ and $Y=\{y_j: j \in I\}$ be relatively separated and $r-$dense sets in $G.$ Assume there exists a family of averaging functions $\{\eta_{i,j}: i,j \in I\}$ satisfying 
    \begin{itemize}
        \item $0 \leq \eta_{i,j} \leq 1$ for all $i,j \in I,$
        \item $\displaystyle\int_{G \times G} \eta_{i,j} =1$ for all $i,j \in I,$ and
        \item for all $i,j \in I,$ $\operatorname{supp}(\eta_{i,j}) \subset B(x_i,r) \times B(y_j,r).$
    \end{itemize}
    If there is an $r>0$ such that 
    \begin{equation}
        \|\operatorname{osc}_{2r} (K)\|_{CS} < 1/2,
    \end{equation}
    then any signal $f \in V$ can be recovered from its averages $$\displaystyle \left\{\langle f,\eta_{i,j} \rangle=\int_{G\times G} f(x,y) \eta_{i,j}(x,y) ~d\mu(x)d\mu(y): i ,j\in I\right\}.$$ Moreover, the following average sampling inequality 
    \begin{multline}
    \frac{1}{L_X L_Y} (1-2\|\operatorname{osc}_{2r}(K)\|_{CS}) \|f\|_{L^{\phi,\psi}_{1/w}(G \times G)} \lesssim \|(\langle f,\eta_{i,j} \rangle)_{i,j \in I}\|_{\ell^{\phi,\psi}_{1/w}(X \times Y)} \\ \lesssim  L_X^2 L_Y^3 \|f\|_{L^{\phi,\psi}_{1/w}(G \times G)}\left(  1+ 2\|\operatorname{osc}_{2r}(K)\|_{CS}\right)
        \end{multline}
        holds for all $f \in V.$    
\end{theorem}
\begin{proof}
    The proof follows along the lines of Theorem \ref{avgsamplingorlicz} and Theorem \ref{samplingmixedorlicz}. The only modification is the definition of the average sampling operator, which is
    $$ \displaystyle \mathcal{S}_{avg}f(x,y):= \sum_{i,j \in I}\langle f,\eta_{i,j} \rangle \frac{\chi_{B(x_i,r)}(x)}{N_X(x)} \frac{\chi_{B(y_j,r)}(y)}{N_Y(y)}.$$ \qedhere
\end{proof}
\subsection{Random sampling of signals in mixed-norm Orlicz spaces}
As in the case of Orlicz spaces, we impose additional regularity conditions on the integral kernel $K$ for a random sampling result. We assume that 
\begin{equation}
\label{mixedkernelregularity}
    \tilde{\kappa}_{\theta}= \|K\|_{CS} + \sup_{0 < r \leq 1} r^{-\theta}\|\operatorname{osc}_{r}(K)\|_{CS} < \infty
\end{equation}
for some $\theta \in (0,1].$ This assumption in particular ensures that the integral operator is bounded on $L^{\phi,\psi}_{1/w}(G \times G),$ among some other technicalities. \par
The set of essentially concentrated functions is defined as 
\begin{equation}
    V_{\delta}=\{ f \in V: \|f\|_{L^{\phi,\psi}_{1/w}(G_N \times G_M)} \geq (1-\delta) \|f\|_{L^{\phi,\psi}_{1/w}(G \times G)} \}.
\end{equation} 
The following theorem concerns the sampling of essentially concentrated functions in terms of the \emph{covering radius}. Let $A,B \subset G.$ The covering radius of $X\times Y$, where $X=\{x_1, x_2, \ldots x_n\} \subset A$ and $Y=\{y_1, y_2, \ldots, y_m\} \subset B$ for $A \times B$ is defined as
\begin{equation}
    r_{X \times Y} (n,m)= \inf \left\{ r> 0: A \times B \subset \bigcup_{i,j=1}^{n,m} B(x_i,r) \times B(y_j,r) \right\}.
\end{equation}
\begin{lemma}
\label{mixedsamplingcoveringradius}
    Let $T$ be an idempotent integral operator on $L^{\phi,\psi}_{1/w}(G \times G)$ with integral kernel $K$ satisfying \eqref{mixedkernelregularity}, (A1) and (A2). Assume that $X=\{x_1,x_2,\ldots, x_n\}$ and $Y=\{y_1,y_2,\ldots,y_m\}$ are finite collection of sample points in $G_N$ and $G_M$ respectively. If the covering radius of $X \times Y$ satisfies 
    \begin{equation}
    \label{crossradiusregularityrelation}
        r_{X\times Y}(n,m)<\left(\frac{1-\delta}{2\tilde{\kappa}_\theta}\right)^{1/\theta},
    \end{equation}
    then 
    \begin{multline}
   (1-\delta-2r_{X\times Y}(n,m)^{\theta} \tilde{\kappa}_{\theta})\|f\|_{L^{\phi,\psi}_{1/w}(G \times G)} \leq \|(f(x_i,y_j))_{i,j=1}^{n,m}\|_{\ell^{\phi,\psi}_{1/w}(X \times Y)} \\ \leq  \frac{1}{\mu(B(e,\zeta))^2} \|f\|_{L^{\phi,\psi}_{1/w}(G \times G)} (1+ r_{X\times Y}(n,m)^{\theta} \tilde{\kappa}_{\theta})
\end{multline}
    for all $f \in V_{\delta},$ where $\zeta$ depends on the separating distance between the sampling points $X \times Y.$
\end{lemma}
\begin{proof}
    Let $X$ and $Y$ be collection of sample points as above satisfying \eqref{crossradiusregularityrelation}. Define the sampling operator 
    \begin{equation*}
        \mathcal{A}f(x,y):=\sum_{i,j=1}^{n,m}  f(x_i,y_j) \chi_{B_{i}}(x) \chi_{\widetilde{B}_{j}}(y) 
    \end{equation*}
    for $x \in G_N,$ $y \in G_M,$ and $f \in V_{\delta},$ where $\{B_i : i=1,2,\ldots ,n\}$ is a collection extracted from $\{B(x_i,r): i=1,2,\ldots,n\},$ and $\{\widetilde{B}_j : j=1,2,\ldots ,m\}$ is a collection extracted from $\{B(y_j,r): j=1,2,\ldots,m\}$ similar to Lemma \ref{Orliczcoveringradiussampling}. \par 
    Now, using the Lemma \ref{mixedseqequiv}, but restricting to the set $G_N \times G_M,$ we get that 
 \begin{equation*}
     \|\mathcal{A}f\|_{L^{\phi,\psi}(G_N \times G_M)} \leq  \|(f(x_i,y_j))_{i,j=1}^{n,m}\|_{\ell^{\phi,\psi}_{1/w}(X \times Y)}.
 \end{equation*}
    Next, it can be observed that 
    \begin{equation*}
        |\mathcal{A}f(x,y)-f(x,y)| \leq \operatorname{osc}_{r}(f)(x,y) ~~ \mbox{for } (x,y) \in G_N \times G_M
    \end{equation*}
    and therefore 
    \begin{equation*}
        \|\mathcal{A}f-f\|_{L^{\phi,\psi}_{1/w}(G_N \times G_M)} \leq \|\operatorname{osc}_{r}(f)\|_{L^{\phi,\psi}_{1/w}(G_N \times G_M)} \leq 2\|\operatorname{osc}_{r}(K)\|_{CS}  \|f\|_{L^{\phi,\psi}_{1/w}(G \times G)}.
    \end{equation*}
Consequently, 
\begin{align*}
    \|(f(x_i,y_j))_{i,j=1}^{n,m}\|_{\ell^{\phi,\psi}_{1/w}(X \times Y)} &\geq \|\mathcal{A}f\|_{L^{\phi,\psi}_{1/w}(G_N \times G_M)} \\
    & \geq \|f\|_{L^{\phi,\psi}_{1/w}(G_N \times G_M)} -\|\mathcal{A}f-f\|_{L^{\phi,\psi}_{1/w}(G_N \times G_M)} \\
    &\geq (1-\delta)\|f\|_{L^{\phi,\psi}_{1/w}(G \times G)} - 2\|\operatorname{osc}_{r}(K)\|_{CS}  \|f\|_{L^{\phi,\psi}_{1/w}(G \times G)} \\
    &= (1-\delta-2\|\operatorname{osc}_{r}(K)\|_{CS})\|f\|_{L^{\phi,\psi}_{1/w}(G \times G)} \\
    &\geq (1-\delta-2 r^{\theta}\tilde{\kappa}_{\theta}) \|f\|_{L^{\phi,\psi}_{1/w}(G \times G)}.
\end{align*}
The above sampling inequality follows verbatim, except for some minor changes that are easily accommodated, from Theorem \ref{samplingmixedorlicz}. The only additional thing is the estimate on the oscillation of the kernel, which is used at the last step of the argument.
\end{proof}
\begin{theorem}
    Let $T$ be an idempotent integral operator on $L^{\phi,\psi}_{1/w}(G \times G)$ with integral kernel $K$ satisfying the condition (\ref{mixedkernelregularity}), (A1) and (A2). Assume $X=\{x_1,x_2,\ldots,x_n\}$ are i.i.d. random variables distributed uniformly over the set $G_N,$ and $Y=\{y_1,y_2, \ldots, y_m\}$ are i.i.d. random variables distributed uniformly over the set $G_M.$  Then for $0 <\epsilon< 1-\delta$
\begin{equation}
    (1-\delta-\epsilon)\|f\|_{L^{\phi,\psi}_{1/w}(G \times G)} \leq \|(f(x_i,y_j))_{i,j=1}^{n,m}\|_{\ell^{\phi,\psi}_{1/w}(X \times Y)} \leq c_{n,m} D_{N,M}\|f\|_{L^{\phi,\psi}_{1/w}(G \times G)}
\end{equation}
holds for all $f \in V_{N,\delta}$ with probability atleast 
\begin{multline}
   1- \frac{\mu(G_N) \mu(G_M)}{\mu(B(e,\epsilon^{1/\theta} ({2\tilde{\kappa}_{\theta}})^{-1/\theta}/4))^2}\left\{ \left(1- \frac{\mu(B(e,\epsilon^{1/\theta} ({2\tilde{\kappa}_{\theta}})^{-1/\theta}/2))}{\mu(G_N)}\right)^{n} + \right. \\ \left.  \left[ 1-  \left(1- \frac{\mu(B(e,\epsilon^{1/\theta} ({2\tilde{\kappa}_{\theta}})^{-1/\theta}/2))}{\mu(G_N)}\right)^{n}\right]  \left(1- \frac{\mu(B(e,\epsilon^{1/\theta} ({2\tilde{\kappa}_{\theta}})^{-1/\theta}/2))}{\mu(G_M)}\right)^{m} \right\},
\end{multline}
where $D_{N,M}=\sup_{(a,b) \in G_N \times G_M} \|K((a,b),(\cdot,\cdot)) \|_{L^{\phi^*, \psi^*}_{w}(G \times G)}$ and $c_{n,m}$ is a constant depending on $n$ and $m$.
\end{theorem}

\begin{proof}
    Let $X=\{x_1,x_2,\ldots,x_n\},$ $Y=\{y_1,y_2, \ldots, y_m\}$ be  collections of i.i.d. random variables uniformly distributed over the sets $G_N$ and $G_M$ respectively, $\alpha > 0$ and define the event $E=\{r_{X \times Y}(n,m) \geq \alpha\}.$ By Lemma \ref{mixedsamplingcoveringradius}, if the covering radius of $X \times Y$ satisfies
$$r_{X\times Y}(n,m)<\left(\frac{1-\delta}{2\tilde{\kappa}_\theta}\right)^{1/\theta},$$
then $X \times Y$ is a stable sampling set for $V_{\delta}$. We will estimate the probability of the event $E$ which in turn will provide a lower bound for the probability of sampling inequality to hold. \par 
Let $\Gamma_1 \subset G_N$ be an $\alpha/2$-net of $G_N$, whose cardinality does not exceed $\frac{\mu(G_N)}{\mu(B(e,\alpha/4))},$ and $\Gamma_2 \subset G_M$ be an $\alpha/2$-net of $G_M$, whose cardinality does not exceed $\frac{\mu(G_M)}{\mu(B(e,\alpha/4))}.$ Then the collection 
$$ \mathcal{C}= \{B(\gamma, \alpha/2) \times B(\gamma', \alpha/2): \gamma \in \Gamma_1, \gamma' \in \Gamma_2\}.$$
forms a cover for $G_N \times G_M.$
Assume that the event $E$ is true for some $0 < \alpha < 1.$ Since $r_{X \times Y} (n,m) \geq \alpha,$ there is a point $(\gamma , \gamma') \in \Gamma_1 \times \Gamma_2 $ such that no sampling point from $X \times Y$ lies in the set $B(\gamma, \alpha/2) \times B(\gamma', \alpha/2).$ Thus 
$$E \subset \left\{ \exists ~ (\gamma, \gamma') \in \Gamma_1 \times \Gamma_2: B(\gamma, \alpha/2) \times B(\gamma', \alpha/2) \cap (X \times Y) = \varnothing  \right\}.$$
Now, the probability that for a fixed $(\gamma, \gamma') \in \Gamma_1 \times \Gamma_2$, the intersection $(B(\gamma, \alpha/2) \times B(\gamma', \alpha/2)) \cap (X \times Y) = \varnothing$ is
\begin{multline*}
\mathbb{P}(\{B(\gamma, \alpha/2) \times B(\gamma', \alpha/2) \cap (X \times Y) = \varnothing\})= \left(1- \frac{\mu(B(e,\alpha/2))}{\mu(G_N)}\right)^{n} + \\ \left[ 1-  \left(1- \frac{\mu(B(e,\alpha/2))}{\mu(G_N)}\right)^{n}\right]  \left(1- \frac{\mu(B(e,\alpha/2))}{\mu(G_M)}\right)^{m}  
\end{multline*} 
Now, using the union bound 
\begin{align*}
    \mathbb{P}(E) \leq \frac{\mu(G_N) \mu(G_M)}{\mu(B(e,\alpha/4))^2}\left\{ \left(1- \frac{\mu(B(e,\alpha/2))}{\mu(G_N)}\right)^{n} + \right. \\ \left.  \left[ 1-  \left(1- \frac{\mu(B(e,\alpha/2))}{\mu(G_N)}\right)^{n}\right]  \left(1- \frac{\mu(B(e,\alpha/2))}{\mu(G_M)}\right)^{m} \right\}.
\end{align*}
By Lemma \ref{mixedsamplingcoveringradius}, the sampling inequality holds true if $r_{X \times Y}(n,m) < \left(\frac{1-\delta}{2\tilde{\kappa}_{\theta}}\right)^{1/\theta}.$ In particular, if we put $\alpha=\epsilon^{1/\theta} ({2\tilde{\kappa}_{\theta}})^{-1/\theta},$ the sampling inequality holds with probability at least
\begin{multline*}
   1- \frac{\mu(G_N) \mu(G_M)}{\mu(B(e,\epsilon^{1/\theta} ({2\tilde{\kappa}_{\theta}})^{-1/\theta}/4))^2}\left\{ \left(1- \frac{\mu(B(e,\epsilon^{1/\theta} ({2\tilde{\kappa}_{\theta}})^{-1/\theta}/2))}{\mu(G_N)}\right)^{n} + \right. \\ \left.  \left[ 1-  \left(1- \frac{\mu(B(e,\epsilon^{1/\theta} ({2\tilde{\kappa}_{\theta}})^{-1/\theta}/2))}{\mu(G_N)}\right)^{n}\right]  \left(1- \frac{\mu(B(e,\epsilon^{1/\theta} ({2\tilde{\kappa}_{\theta}})^{-1/\theta}/2))}{\mu(G_M)}\right)^{m} \right\}.
\end{multline*}
\end{proof} 
In the above theorem, one can, in particular, take 
$$c_{n,m}=\frac{4}{\sqrt{w(e,e)}}  \|\underbrace{(1,1,\cdots,1)}_{n-terms}\|_{\ell^{\phi}(X)} \|\underbrace{(1,1,\cdots,1)}_{m-terms}\|_{\ell^{\psi}(Y)}.$$

\section{Examples} \label{examples}
\subsection{Sampling in weighted shift-invariant subspaces of Orlicz spaces} For an Orlicz function $\phi,$ we want to consider sampling for the functions in the shift-invariant space 
$$\displaystyle V_{\phi, 1/w}(g):=\left\{ f= \sum_{k \in \ZZ^n} c_k g(\cdot -k): (c_{k})_{k \in \ZZ^n} \in \ell^{\phi}_{1/w}(\ZZ^n) \right\}, $$
for a generator function $g$ chosen appropriately. 

The first issue here is the well-definedness of these shift-invariant spaces. This requires establishing the convergence of the series generated by translates of the underlying generator and showing that the resulting functions belong to the weighted Orlicz space. To study the well-definedness, we begin by introducing a class of functions to which the generator of the shift-invariant space will belong. Following the terminology in \cite{UnserSamplingNOnDecaying}, these spaces are referred to as \emph{weighted hybrid-norm spaces.}

\begin{definition}[Weighted hybrid-Norm spaces]
    For $p,q \geq 1,$ the hybrid-norm spaces $W_{p,q}(\RR^{n})$ consists of functions $f: \RR^n  \to \CC$ such that the norm
    \begin{equation*}
        \|f\|_{W_{p,q}(\RR^n)}:= \left( \int_{\RR^n} \left\{ \sum_{k \in \ZZ^n}|f(x+k)|^{p} \right\}^{q/p} dx\right)^{1/q}
    \end{equation*}
    is finite, with usual adjustment for the case when either of $p$ or $q$ is infinity. For a weight function $w,$ the weighted hybrid norm space $W_{p,q,w}(\RR^{n})$ consists of functions $f$ having the the norm
    \begin{equation*}
        \|f\|_{W_{p,q,w}(\RR^n)}:= \|fw\|_{W_{p,q}(\RR^n)}
    \end{equation*}
    is finite.
\end{definition}
We will need the following result, due to Nyugen and Unser, from \cite{UnserSamplingNOnDecaying}.
\begin{prop}[{\cite[Proposition 6]{UnserSamplingNOnDecaying}}] \label{GRScond}
    Let $w$ be a sub-multiplicative weight satisfying the Gelfand-Raikov-Shilov (GRS) condition
    \begin{equation}
        \lim_{n \to \infty} w(kn)^{1/n}=1, ~ \forall k \in \ZZ^n.
    \end{equation}
    Suppose that the generator $g \in W_{1,\infty,w}(\RR^n),$ and $\{g(\cdot-k): k \in \ZZ^n\}$ is a Riesz basis for $V_{2}(g).$ Then the dual generator $g_{\operatorname{\tiny dual}}$ is also in $W_{1,\infty,w}(\RR^n).$
\end{prop}
A fundamental ingredient in the study of well-definedness of the shift-invariant spaces is the stability of the system of translates of the generator. This stability is expressed generally through Riesz-type bounds, which guarantee that the coefficient sequence and the corresponding function are determined by each other in a stable way. The following will establish these bounds in the present setting.    
\begin{theorem} \label{orliczsis}
    Let $w$ be a sub-multiplicative weight satisfying the GRS condition, and further assume that $g$ satisfies the assumptions in Proposition \ref{GRScond}, then the following inequality holds 
    \begin{equation}
        C_1 \|c\|_{\ell^{\phi}_{1/w}(\ZZ^n)} \leq \left\| \sum_{k \in \ZZ^n} c_k  g(\cdot-k) \right\|_{L^{\phi}_{1/w}(\RR^n)} \leq C_2 \|c\|_{\ell^{\phi}_{1/w}(\ZZ^n)}, ~~ \forall c \in \ell^{\phi}_{1/w}(\ZZ^n)
    \end{equation}
  and for some $C_1, C_2 > 0.$

\end{theorem}
 \begin{proof}
      To see this, note that for $\displaystyle f=\sum_{k \in \ZZ^n} c_k g(\cdot -k)$ 
    \begin{equation*}
        \begin{split}
            \frac{|f(x)|}{w(x)} &\leq \sum_{k \in \ZZ^n} \frac{|c_k| |g(x-k)|}{w(x)} \\
            &\leq \sum_{k \in \ZZ^n} \frac{|c_k|}{w(k)} |g(x-k)|w(x-k).
        \end{split}
    \end{equation*}
Since $\displaystyle \|g\|_{W_{1,\infty,w}}= \sup_{x \in [0,1]^n} \sum_{k \in \ZZ^n} |g(x-k)w(x-k)|,$ and observe
    \begin{equation*}
        \phi\left(\frac{|f(x)|}{\lambda w(x)}\right) \leq \sum_{k \in \ZZ^n} \phi \left( \frac{|c_k| \|g\|_{W_{1,\infty,w}}}{w(k) \lambda} \right) \frac{|g(x-k)||w(x-k)|}{\|g\|_{W_{1,\infty,w}}}
    \end{equation*}
    and therefore
    \begin{equation*}
        \int_{\RR^n} \phi\left(\frac{|f(x)|}{\lambda w(x)}\right) d\mu(x) \leq \sum_{k \in \ZZ^n} \phi \left( \frac{|c_k| \|g\|_{W_{1,\infty,w}}}{w(k) \lambda} \right) \frac{\|g\|_{L^{1}_{w}(\RR^n)}}{\|g\|_{W_{1,\infty,w}}}.
    \end{equation*}
    This implies
    \begin{equation*}
        \|f\|_{L^{\phi}_{1/w}(\RR^n)} \leq \|c\|_{\ell^{\phi}_{1/w}(\ZZ^n)} \|g\|_{W_{1,\infty,w}}.
    \end{equation*}

Now, note that since $g$ and $g_{\text{\tiny dual}}$ are orthogonal, therefore $c_k = \langle f, g_{\tiny\mbox{dual}}(\cdot -k) \rangle.$
Take $b=(b_k) \in \ell^{\phi^*}_{w}(\ZZ^n),$ the dual space of $\ell^{\phi}_{1/w}(\ZZ^n)$ with $\|b\|_{\ell^{\phi^*}_{w}(\ZZ^n)}=1,$ and consider 
    \begin{equation*}
        \begin{split}
            \langle b,c \rangle &= \sum_{k \in \ZZ^n} b_k \int_{\RR^n} f(t) \overline{g_{\operatorname{\tiny{dual}}}(k-t)} dt \\
             &= \int_{\RR^n} \frac{f(t)}{w(t)} \sum_{k \in \ZZ^n} b_k \overline{g_{\operatorname{\tiny{dual}}}(k-t)} w(t) dt \\
             &\leq \|f\|_{L^{\phi}_{1/w}(\RR^n)} \left\|\sum_{k \in \ZZ^n} b_k \overline{g_{\operatorname{\tiny{dual}}}(k-\cdot)}\right\|_{L^{\phi^*}_{w}(\RR^n)}\\ 
             &\leq \|f\|_{L^{\phi}_{1/w}(\RR^n)} \|b\|_{\ell^{\phi^*}_{w}(\ZZ^n)} \|g_{\operatorname{\tiny{dual}}}\|_{W_{1,\infty,w}}.
        \end{split}
    \end{equation*}
    This gives 
    \begin{equation*}
        \|c\|_{\ell^{\phi}_{1/w}(\ZZ^n)} \lesssim \|f\|_{L^{\phi}_{1/w}(\RR^n)}. \qedhere
    \end{equation*}
     \end{proof}
As an immediate consequence of the preceding result, the space $V_{\phi,1/w}(g)$ is well defined and it is the shift-invariant space generated by $g$. That is, every function in the shift-invariant space can be represented as a convergent series of translates of the generator $g$, with coefficients belonging to the corresponding weighted Orlicz sequence space.

Further, it can be noted that the weighted shift-invariant space $V_{\phi,1/w}(g)$ is the image space of an idempotent integral operator with the kernel given by
$$K(x,y)=\sum_{k\in\mathbb{Z}^n}g(x-k)\,g_{\text{\tiny{dual}}}(y-k).$$
Consequently, if both $g$ and $g_{\mathrm{dual}}$ possess suitable decay, the sampling theorems established guarantee stable reconstruction of functions in $V_{\phi,1/w}(g)$.
In \cite{UnserSamplingNOnDecaying,NondecayingMixedLebesgue}, the authors studied the stable sampling problem only in the setting where the samples are obtained through pre-filtering. Consequently, their work is restricted to this specific sampling framework in $V_{\phi,1/w}$. In contrast, our result establishes a broader class of sampling problems in $V_{\phi,1/w}$, including direct sampling, average sampling, and random sampling.

     
  \subsection{Sampling in weighted shift-invariant subspaces of mixed-norm Orlicz spaces } For Orlicz functions $\phi$ and $\psi,$ we now define the weighted mixed-norm shift-invariant spaces 
  $$\displaystyle V_{\phi,\psi, 1/w}(g):=\left\{ f= \sum_{k,l \in \ZZ^n} c_{k,l} g(\cdot-k,\cdot-l): (c_{k,l})_{k,l \in \ZZ^n} \in \ell^{\phi,\psi}_{1/w}(\ZZ^n \times \ZZ^n)  \right\}, $$
  for an appropriately chosen generator $g$. As in the previous subsection, we begin by proving the well-definedness of these spaces. In this direction, we first establish that under certain decay assumption on the generator $g$, the expansion  $\sum_{k,l \in \ZZ^n} c_{k,l} g(\cdot-k,\cdot-l)$ makes sense in the weighted mixed-norm Orlicz space. This is the content of the following result.
     \begin{lemma}\label{mixedsisupper}
        Let $\phi, \psi$ be Orlicz functions and $w$ be a sub-multiplicative weight. Assume $g:\RR^n \times \RR^n \to \CC$ is a function such that 
        \begin{equation} \label{sisconditionformixed}
            \sup_{y \in [0,1]^n} \sum_{k \in \ZZ^n} \sup_{x \in [0,1]^n} \sum_{l \in \ZZ^n} |g(x-k,y-l)| w(x-k,y-l)< \infty,
        \end{equation}
        then $\displaystyle f= \sum_{k,l \in \ZZ^n} c_{k,l} g(\cdot-k,\cdot-l) $ belongs to $ L^{\phi, \psi}_{1/w}(\RR^n \times \RR^n)$ for $c=(c_{k,l})_{k,l \in\ZZ^n} \in \ell^{\phi,\psi}_{1/w} (\ZZ^n \times \ZZ^n),$ and the estimate 
        \begin{equation}
            \|f\|_{L^{\phi,\psi}_{1/w}(\RR^n \times \RR^n)} \leq C \|c\|_{\ell^{\phi,\psi}_{1/w}(\ZZ^n \times \ZZ^n)}
        \end{equation}
        holds for some constant $C>0$.
     \end{lemma}
     \begin{proof}
         Let $\displaystyle f(x,y)=\sum_{k,l \in \ZZ^n} c_{k,l} g(x-k,y-l),$ $c=(c_{k,l})_{k,l \in \ZZ^n} \in \ell^{\phi,\psi}_{1/w}(\ZZ^n \times \ZZ^n),$ then 
         \begin{equation*}
             \begin{split}
                 \frac{|f(x,y)|}{w(x,y)} &\leq \sum_{k,l \in \ZZ^n} \frac{|c_{k,l}| |g(x-k,y-l)|}{w(x,y)} \\
                 & \leq \sum_{k,l \in \ZZ^n} \frac{|c_{k,l}|}{w(k,l)} |g(x-k,y-l)| w(x-k,y-l).
             \end{split}
         \end{equation*}
         Since for each $y$ and $n,$  
         $$\displaystyle \sup_{x \in [0,1]^{n}}\sum_{k \in \ZZ^n} |g(x-k,y-l)|w(x-k,y-l) < \infty$$
         using triangle inequality and Lemma \ref{orliczsis}, we get
         \begin{equation*} 
         \begin{split}
             \left\|\frac{|f(\cdot,y)|}{w(\cdot,y)}\right\|_{L^{\phi}(\RR^n)} &\leq \sum_{n \in \ZZ^n} \left\| \sum_{k \in \ZZ^n} \frac{|c_{k,l}|}{w(k,l)} |g(x-k,y-l)| w(x-k,y-l) \right\|_{L^{\phi}(\RR^n)} \\
             &\leq \sum_{l \in \ZZ^n} \left\|\left(\frac{c_{k,l}}{w(k,l)}\right)_{m \in \ZZ^n} \right\|_{\ell^{\phi}(\ZZ^n)} \|g(\cdot, y-l)w(\cdot,y-l)\|_{W_{1,\infty}}.
             \end{split}
         \end{equation*}
        Observe (\ref{sisconditionformixed}) also implies that 
        $$\displaystyle \sup_{y \in [0,1]^n} \sum_{l \in \ZZ^n} \|g(\cdot, y-l) w(\cdot, y-l)\|_{W_{1,\infty}} < \infty$$
        and therefore 
         \begin{equation*}
             \left\|\frac{|f(\cdot,\cdot)|}{w(\cdot,\cdot)}\right\|_{L^{\phi,\psi}(\RR^n \times \RR^n)} \lesssim \left\|\left(\frac{c_{k,l}}{w(k,l)}\right)_{k,l \in \ZZ^n}\right\|_{\ell^{\phi,\psi}(\ZZ^n \times \ZZ^n)} ,
         \end{equation*}
         which, in different notations, is nothing but
         \begin{equation*}
             \|f\|_{L^{\phi,\psi}_{1/w}(\RR^n \times \RR^n)} \lesssim \|c\|_{\ell^{\phi,\psi}_{1/w}(\ZZ^n \times \ZZ^n)}.
         \end{equation*}  
         This concludes the proof.
     \end{proof}
     The following result establishes the stability of the coefficient sequence obtained through the dual pairing between a function and the translates of a generator function. 
     \begin{lemma}\label{mixedsislower}
         Let $\phi, \psi$ be Orlicz functions, and $w$ be a sub-multiplicative weight. Assume $g:\RR^n \times \RR^n \to \CC$ is a function such that it belongs to the space $L^{\phi^{*},\psi^{*}}_{w}(\RR^n \times \RR^n)$ and satisfies \eqref{sisconditionformixed}. Then the sequence $(c_{k,l})= \left( \langle f , T_{k,l}g \rangle \right)$ belongs to $\ell^{\phi,\psi}_{1/w}(\ZZ^n \times \ZZ^n)$ for $f \in L^{\phi,\psi}_{1/w}(\RR^n \times \RR^n),$ and we have the norm bound
         \begin{equation}
             \|c\|_{\ell^{\phi,\psi}_{1/w}(\ZZ^n \times \ZZ^n)} \leq C \|f\|_{L^{\phi,\psi}_{1/w}(\RR^n \times \RR^n)}
         \end{equation}
         for some constant $C>0.$
     \end{lemma}
     \begin{proof}
         Let $b=(b_{k,l}) \in \ell^{\phi^*,\psi^*}_{w}(\ZZ^n \times \ZZ^n)$ with $\|b\|_{\ell^{\phi^*,\psi^*}_{w}(\ZZ^n \times \ZZ^n)}=1,$ and consider
         \begin{equation*}
             \begin{split}
                 \sum_{k,l \in \ZZ^n} b_{k,l} c_{k,l} &= \sum_{k,l \in \ZZ^n} b_{k,l} \langle f , g(\cdot-k,\cdot-l) \rangle \\
                 &=  \left\langle f , \sum_{k,l \in \ZZ^n}b_{k,l}g(\cdot-k,\cdot-l) \right\rangle \\
                 &\leq \|f\|_{L^{\phi,\psi}_{1/w}(\RR^n \times \RR^n)} \left\| \sum_{k,l \in \ZZ^n}b_{k,l}g(\cdot-k,\cdot-l) \right\|_{L^{\phi^*,\psi^*}_{w}(\RR^n \times \RR^n)}.
             \end{split}
         \end{equation*}
         Now using the arguments similar to previous lemma, we get 
         \begin{equation*}
             \sum_{k,l \in \ZZ^n} b_{k,l} c_{k,l} \lesssim \|f\|_{L^{\phi,\psi}_{1/w}(\RR^n \times \RR^n)} \|b\|_{\ell^{\phi^*,\psi^*}_{w}(\ZZ^n \times \ZZ^n)} .
         \end{equation*}
         This implies 
         \begin{equation*}
              \|c\|_{\ell^{\phi,\psi}_{1/w}(\ZZ^n \times \ZZ^n)} \lesssim \|f\|_{L^{\phi,\psi}_{1/w}(\RR^n \times \RR^n)} 
         \end{equation*}
         and we are done.
     \end{proof}
     \begin{lemma}
         Let $g$ belong to the Schwartz class $\mathcal{S}(\RR^n \times \RR^n),$ and $w$ be a sub-multiplicative weight function that does not grow more than a polynomial. Assume that $\{g(\cdot-k,\cdot-l):  k,l \in \ZZ^n\}$ forms a Riesz basis for the shift invariant space generated by it. Then the dual generator $g_{\operatorname{\tiny{dual}}}$ is also in $\mathcal{S}(\RR^n \times \RR^n),$ and both $g$ and $g_{\operatorname{\tiny{dual}}}$ satisfy \eqref{sisconditionformixed}.
     \end{lemma}
     \begin{proof}
         The fact that $g_{\text{\tiny{dual}}}$ belongs to the $\mathcal{S}(\RR^n \times \RR^n)$ follows from the characterization of the dual generator in the shift-invariant space theory. Moreover, the polynomial decay of arbitrary order of both $g$ and $g_{\text{\tiny dual}}$ ensures (\ref{sisconditionformixed}).
     \end{proof}
     The next theorem is in the same spirit as that of Theorem \ref{orliczsis}. It establish the Riesz-type inequalities for the system $\{g(\cdot-k,\cdot-l): k,l \in \ZZ^n\}.$
     \begin{theorem}
         Let $\phi,\psi$ be Orlicz functions, and $w$ be a sub-multiplicative weight with at most polynomial growth. Assume $g \in \mathcal{S}(\RR^n \times \RR^n),$ and $ \{g(\cdot-k,\cdot-l): k,l \in \ZZ^n \}$ forms a Riesz sequence. Then the following inequality holds 
         \begin{equation}
             C_1 \|c\|_{\ell^{\phi,\psi}_{1/w}(\ZZ^n \times \ZZ^n)} \leq \left\| \sum_{k,l \in \ZZ^n} c_{k,l} g(\cdot-k,\cdot-l) \right\|_{L^{\phi,\psi}_{1/w}(\RR^n \times \RR^n)} \leq C_2 \|c\|_{\ell^{\phi,\psi}_{1/w}(\ZZ^n \times \ZZ^n)}
         \end{equation}
         for some $C_1,C_2 > 0$ and for all $c \in \ell^{\phi,\psi}_{1/w}(\ZZ^n \times \ZZ^n).$
     \end{theorem}
     \begin{proof}
        By the assumptions, the dual generator $g_{\text{\tiny dual}}$ belongs to the space $\mathcal{S}(\RR^n \times \RR^n),$ and therefore to the space $L^{\phi^*, \psi^*}_{w}(\RR^n \times \RR^n).$ Both $g$ and $g_{\text{\tiny dual}}$ satisfy $(\ref{sisconditionformixed}).$ Therefore, Lemma \ref{mixedsisupper} and Lemma \ref{mixedsislower} are applicable. Upper inequality follows from Lemma \ref{mixedsisupper}, and lower inequality follows from Lemma \ref{mixedsislower} upon observing that $c_{k,l}= \left\langle f, g_{\tiny\mbox{dual}}(\cdot-k,\cdot-l)\right\rangle.$ 
    \end{proof}
The previous theorem established the well-definedness of the space $V_{\phi,\psi,1/w}(g).$ It also shows that the shift-invariant space generated $g$ coincides with the space $V_{\phi,\psi,1/w}(g).$ Also, as in the weighted Orlicz space setting, the weighted shift-invariant subspaces of mixed-norm Orlicz spaces are the image space of idempotent integral operators. Consequently, provided the generator and its dual satisfy suitable decay conditions, the sampling theorems developed in this paper guarantee stable reconstruction of functions in these spaces.

\subsection{Sampling in Orlicz-modulation spaces}

 We describe Orlicz modulation spaces using the framework of coorbit theory. The Orlicz spaces were introduced by Feichtinger in a series of papers starting from \cite{FeichtingerSegalAlgebra}. The description of Orlicz spaces can also be given using what is known as the Coorbit theory, introduced by Feichtinger and Gr\"ochenig in \cite{FeichtingerGrochenigCoorbit}. We briefly review the necessary background and refer the reader to \cite{primercoorbit} for a more detailed and accessible exposition, along with further references within.
 \par 
We consider the locally compact Hausdorff group $\mathbb{H}^{n}_{r}= \RR^n \times \RR^n \times \mathbb{T},$ with Haar measure $dx ~ dw ~dt,$ and the group operation is given by
\begin{equation*}
    (x,w,e^{2 \pi i t})(x',w',e^{2 \pi i t'}):= \left(x+x', w+w', e^{2 \pi i (t+t')} e^{\pi i (x'w-xw')}\right).
\end{equation*}
 The group $\mathbb{H}^{n}_{r}$ is called the \emph{reduced Heisenberg group}. The \emph{Schr\"odinger representation} of the reduced Heisenberg group $\rho: \mathbb{H}^{n}_{r} \to \mathcal{U}(L^{2}(\RR^n)),$ where $\mathcal{U}(L^{2}(\RR^n))$ is unitary operators on $L^{2}(\RR^n),$ is given by 
 \begin{equation*}
     \rho(x,w,e^{2 \pi i t})f(y):= e^{2 \pi i t} e^{\pi i xw} T_{x}M_{w}f(y).
 \end{equation*}
This is an irreducible, unitary, and square integrable representation. The \emph{wavelet transform} with respect to the Schr\"odinger representation $\rho$ is given by 
\begin{equation*}
    W_{g}f(x,w,e^{2 \pi i t})= \langle f, \rho\left(x,w,e^{2 \pi i t}\right)g\rangle_{L^{2}(\RR^n)}
\end{equation*}
for $f, g \in L^{2}(\RR^n).$ This definition can be extended, by duality, to windows $g \in M^{1}(\RR^n)$, the Feichtinger algebra, and to distributions $f \in M^{\infty}(\RR^n),$ the dual of the Feichtinger algebra. 

The Orlicz modulation spaces are then defined as 
\begin{equation*}
    M^{\phi}(\RR^n):=\{f \in M^{\infty}(\RR^n) : W_{g}f \in L^{\phi}(\mathbb{H}^{n}_{r})\}
\end{equation*}
for $g \in M^{1}(\RR^n),$ and the defined space is independent of the choice of $g.$ The norm on this space is given by
$$\|f\|_{M^{\phi}(\RR^n)}=\left\|W_{g}(f)\right\|_{L^{\phi}(\mathbb{H}^{n}_{r})}.$$
By the correspondence principle, the space $M^{\phi}(\RR^n)$ is isometrically isomorphic to the space 
$$M_{\phi}(\mathbb{H}_{r}^{n}):=\{F \in L^{\phi}(\mathbb{H}^{n}_{r}): F= F * W_{g}g\}$$
where $*$ is the convolution defined with respect to the group operation.

\begin{remark}
    The space $M_{\phi}(\mathbb{H}_{r}^{n})$ is image of an idempotent integral operator with the integral kernel $K$ given by $K(z,\zeta)=W_{g}g(\zeta^{-1}z).$ Consequently, the sampling theorem~\ref{samplingorlicz} reduces to estimating the oscillation of the function $W_{g}g$. In particular, if
$$\|\operatorname{osc}_{r}(W_{g}g)\|_{L^{1}}\longrightarrow 0\qquad\text{as } r\to 0,$$
then the hypotheses of Theorem~\ref{samplingorlicz} are satisfied, and hence the sampling theorem holds for $M_{\phi}(\mathbb{H}_{r}^{n})$. 
\end{remark} 
\begin{lemma}{\cite[Lemma 3.16]{DahlkeShearletCoorbit}}\label{oscillationcontrol}
    Let $F \in L^{1}(G)$ be such that $\displaystyle \int_{G} \left\|R_x \chi_Q F \right\|_{L^{\infty}(G)} d\mu(x) < \infty,$ where $Q$ is a relatively compact neighbourhood of the identity.  Then $ \operatorname{osc}_{r}(F) \in L^{1}(G),$ and if in addition $F$ is continuous, then  
    \begin{equation*}
        \lim_{r \to 0} \|\operatorname{osc}_{r}(F)\|_{L^{1}(G)}=0. 
    \end{equation*}
\end{lemma}
\begin{remark}
    If $g$ is taken to be a Schwartz class function on $\RR^{n}$, then the hypothesis of above lemma is satisfied with $Q=[0,1]^{2n}\times \mathbb{T}$ for $W_{g}g.$ And therefore, Theorem \ref{samplingorlicz} can be applied to obtain a sampling theorem for this space.
\end{remark}

 \subsection{Sampling in mixed-norm Orlicz-modulation spaces} Next example is that of the mixed-norm Orlicz modulation spaces. To define them, we take the group $\mathbb{H}^{n}_{r} \times \mathbb{H}^{n}_{r},$ which is the cartesian product of the reduced Heisenberg group $\mathbb{H}^{n}_{r}.$ On it, we consider the tensor product of the reduced Schr\"odinger representation $\rho$ 
$$\rho \otimes \rho : \mathbb{H}^{n}_{r} \times \mathbb{H}^{n}_{r} \to \mathcal{U}(L^{2}(\RR^n) \otimes L^{2}(\RR^n)) $$
given by 
$$ (\rho\otimes\rho)(g_1,g_2)= \rho(g_1)\otimes \rho(g_2)$$
for $g_1, g_2 \in \mathbb{H}^{n}_{r}.$ This is an irreducible unitary representation of $\mathbb{H}^{n}_{r} \times \mathbb{H}^{n}_{r}$ on $L^{2}(\RR^n) \otimes L^{2}(\RR^n).$ The action of this representation on elementary tensors is given by 
$$ (\rho\otimes\rho)(g_1,g_2) (f \otimes h)= \rho(g_1)(f)\otimes \rho(g_2)(h)$$
for $f,h \in L^{2}(\RR^n),$ and $g_1,g_2 \in \mathbb{H}^{n}_{r}.$ This combined with the fact that $L^{2}(\RR^n) \otimes L^{2}(\RR^n)$ is isometrically isomorphic to $L^{2}(\RR^n \times \RR^n),$ the action of $\rho \otimes \rho$ on any arbitrary $F \in L^{2}(\RR^n \times \RR^n)$ is given by 
\begin{multline*}
    (\rho \otimes \rho) ((x_1,w_1,t_1),(x_2,w_2,t_2))(F)(y_1,y_2)=e^{2 \pi i (t_1 +t_2)} e^{\pi i (x_1 w_1 + x_2 w_2)} \\T_{(x_1,x_2)}M_{(w_1,w_2)} F(y_1,y_2).
\end{multline*}
The mixed-norm Orlicz modulation space is then defined as 
\begin{equation*}
    M^{\phi,\psi}(\RR^{2n})=\{f \in M^{\infty}(\RR^{2n}): W_{g}f \in L^{\phi,\psi}(\mathbb{H}^{n}_{r} \times \mathbb{H}^{n}_{r})\}
\end{equation*}
where $g$ is an element from $M^{1}(\RR^{2n})$, and $W$ is the wavelet transform induced by the representation $\rho \otimes \rho$. The norm on this space is given by 
$$\|f\|_{M^{\phi,\psi}(\RR^{2n})}=\|W_{g}f\|_{L^{\phi,\psi}(\mathbb{H}^{n}_{r} \times \mathbb{H}^{n}_{r})}.$$
The space $M^{\phi,\psi}(\RR^{2n})$ is isometrically isomorphic to 
$$M_{\phi,\psi}(\RR^{2n}):=\{F \in L^{\phi,\psi}(\RR^{2n}): F=F * W_{g}g\}$$
where $*$ is the convolution.

The space $M_{\phi,\psi}(\RR^{2n})$ is image space of idempotent integral operator with kernel
$$K((a,b),(x,y))=W_{g}g((x,y)^{-1} (a,b))=W_{g}g(x^{-1}a,y^{-1}b).$$
The sampling theorem for the space $M_{\phi,\psi}(\RR^{2n})$ can be ensured by considering a nicely decaying function $g$, for e.g. a Schwartz function, and then the decay of the Cross-Schur norm for the oscillation of the kernel $K$ is then ensured by Lemma \ref{oscillationcontrol}. 

\bibliographystyle{abbrv}\bibliography{OMO.bib}
\end{document}